\documentclass[11pt]{amsart}

\usepackage{eucal,graphicx,amssymb,mathrsfs, enumerate}
\usepackage[all, knot]{xy}
\usepackage[utf8]{inputenc}
\DeclareUnicodeCharacter{00A0}{ }
\DeclareUnicodeCharacter{00A0}{~}

\theoremstyle{plain}
\newtheorem{thm}{Theorem}[section]
\newtheorem{prop}[thm]{Proposition}
\newtheorem*{cor*}{Corollary}
\newtheorem{cor}[thm]{Corollary}
\newtheorem{lem}[thm]{Lemma}

\newtheorem*{conj*}{Conjecture}
\newtheorem{mthm}{Theorem}

\newtheorem{theoprime}{Theorem}

\newtheorem{theoseconde}{Theorem}

 \newtheorem{conj}{Conjecture}

\newtheorem*{DMMpbm}{Dynamical Manin-Mumford Problem}

\theoremstyle{definition}

\newtheorem*{ex*}{Example}

\theoremstyle{remark}
\newtheorem{rem}[thm]{Remark}

\numberwithin{equation}{section}

\DeclareMathOperator{\sk}{Sk}
\DeclareMathOperator{\MA}{MA}
\DeclareMathOperator{\reg}{reg}

\DeclareMathOperator{\id}{id}

\DeclareMathOperator{\supp}{supp}

\DeclareMathOperator{\ord}{ord}

\DeclareMathOperator{\spec}{Spec}

\DeclareMathOperator{\aut}{Aut}
\DeclareMathOperator{\an}{an}

\DeclareMathOperator{\geo}{Geo}
\DeclareMathOperator{\jac}{Jac}
\DeclareMathOperator{\loc}{loc}
\DeclareMathOperator{\geom}{\dot{\wedge}}
\DeclareMathOperator{\card}{Card}
\DeclareMathOperator{\fix}{Fix}
\DeclareMathOperator{\per}{Per}

\def\C{\mathbb{C}}
\def\Q{\mathbb{Q}}
\def\P{\mathbb{P}}
\def\R{\mathbb{R}}
\def\Z{\mathbb{Z}}
\def\P{\mathbb{P}}

\def\N{\mathbb{N}}
\def\LL{\mathbb{L}}

\def\A{{\mathbb{A}}}
\def\one{{\mathbf{1}}}

\def\O{{\mathcal{O}}}

\def\cM{{\mathcal{M}}}

\def\n0{{\bf n_0}}

\def\la{{\lambda}}

\providecommand{\norm}[1]{\lVert#1\rVert}

\def\dto{\dashrightarrow}

\newcommand{\e}{\varepsilon}
\newcommand{\cv}{\rightarrow}
\newcommand{\fr}{\partial}
\newcommand{\om}{\Omega}
\newcommand{\set}[1]{\left\{#1\right\}}
 \newcommand{\abs}[1]{\left\vert#1\right\vert}
\newcommand{\cd}{{\mathbb C^2}}
\newcommand{\rest}[1]{ \arrowvert_{#1}}
 
\newcommand{\bb}{\mathbb{B}}
\newcommand{\lrpar}[1]{\left(#1\right)}
 \newcommand{\hot}{\mathrm{h.o.t.}}
\newcommand{\cst}{\mathrm{C^{st}}}

\begin{document}

\title
[The dynamical Manin-Mumford problem]
{The dynamical Manin-Mumford problem for plane polynomial automorphisms}

\author{R. Dujardin}
\address{Laboratoire d'Analyse et de Math\'ematiques Appliqu\'ees,
Universit\'e Paris-Est Marne-la-Vall\'ee,
5, boulevard Descartes,
77454 Champs-sur-Marne}
\email{romain.dujardin@u-pem.fr}

\author{C. Favre}
\address{CNRS - Centre de Math\'ematiques Laurent Schwartz, 
\'Ecole Polytechnique, 
91128 Palaiseau Cedex, France}
\email{favre@math.polytechnique.fr}

\thanks{The second author is supported by the ERC-starting grant project "Nonarcomp" no.307856. The first and second authors are supported by the ANR project ANR-13-BS01-0002.}

\begin{abstract}
Let $f$ be  a polynomial automorphism  of the affine plane. 
In this paper we consider the possibility for it 
  to possess infinitely many periodic points on an algebraic curve $C$.
We conjecture that this happens if and only if 
$f$ admits a time-reversal  symmetry; in particular the Jacobian $\mathrm{Jac}(f)$ must be a root of unity.  

As a step towards this conjecture, we prove that 
 its Jacobian, together with
all its Galois conjugates  lie on the unit circle in the complex plane. 
Under mild additional assumptions we are able to conclude that indeed $\mathrm{Jac}(f)$ is a root of unity.

 We  use these results to show in various cases that any two automorphisms sharing an infinite set of periodic points must have
 a common iterate, in the spirit of recent results by Baker-DeMarco and Yuan-Zhang.
\end{abstract}

\maketitle

 \tableofcontents
\newpage

\section*{Introduction}

In this paper we discuss the following problem in the case of polynomial automorphisms of the affine plane. 

\begin{DMMpbm}
Let $X$ be a quasi-projective variety 
and $f : X \to X$ be a dominant endomorphism. 

Describe all positive-dimensional irreducible subvarieties $C \subset X$ such that 
the Zariski closure of the set of preperiodic\footnote{that is, satisfies $f^n(p) = f^m(p)$ for some $n>m\ge0$.} 
points of $f$ contained in $C$ is Zariski-dense in $C$.
\end{DMMpbm}

In  case $(X,f)$ is the dynamical system  induced on an abelian variety (defined over a number field)   
 by the multiplication by an integer $\ge2$, 
it is a deep theorem originally due to M. Raynaud (and formerly known as the Manin-Mumford conjecture)
 that any such $C$ is   a translate of an abelian subvariety by a torsion point. 
 Several generalizations of this theorem have appeared since then, concerning 
  abelian and semi-abelian varieties over fields of    arbitrary characteristic. 
  We refer to ~\cite{PR02,Roes08} for an account on the different approaches to these results. 

S.-W.~Zhang~\cite{Z95}  conjectured that a similar result should hold in the more general  setting 
 of polarized\footnote{This means that 
 $X$ is projective, and $f^*L \simeq L^{\otimes q}$ for some ample line bundle $L\to X$ and an integer $q
 \ge2$.} endomorphisms. More precisely, he asked whether any subvariety containing a Zariski dense set of periodic points  is itself  
 preperiodic. This conjecture was  recently disproved by D. Ghioca, T. Tucker and S.-W.~Zhang~\cite{GTZ}, who proposed a modified 
 statement (see also~\cite{pazuki}). 
 Some  positive results on Zhang's conjecture are also available, see for instance~\cite{medvedev-scanlon}.

\medskip

Our goal is to explore this problem   when $f$ is a 
 polynomial automorphism of the affine plane $\mathbb{A}^2$,
defined over   a field of characteristic zero.  

Let us first collect  a few facts on the dynamics of these maps. 
A dynamical classification of polynomial automorphisms was given by S. Friedland and J. Milnor 
\cite{friedland-milnor}, based on  a famous theorem of H. W. E. Jung. 
They proved that any polynomial automorphism is conjugate  to one of the following forms:
\begin{itemize}
\item an affine map, 
\item
an elementary automorphism, that is a map of the form $(x,y) \mapsto (ax +b, y + P(x))$  with $a \neq 0$, $b$ is a constant and $P$ is a polynomial, 
\item
a polynomial automorphism $f$ satisfying $\deg(f^n) = \deg(f)^n\ge 2$ for every integer  $n\ge1$.
\end{itemize}
 In the last case the integer $\deg(f)\ge 1$ denotes the maximum of the  degrees of the components of $f$ in any set of 
 affine coordinates. An automorphism   falling into this category  will be referred to as  of   {\em   H\'enon type}. 
Since the dynamical Manin-Mumford problem is uninteresting  for affine and elementary mappings, we  
restrict our attention to  H\'enon-type automorphisms. 
 
 \medskip

Suppose that  $f$ is an automorphism of H\'enon type that is conjugate to its inverse by an involution $\sigma$ possessing a curve $C$
of fixed points. Such a map is usually called reversible, see~\cite{gm1, gm2}.
Then any  point $p\in C\cap f^{-n}(C)$ is periodic of period $2n$, and we verify in \S\ref{sec:reversible}
that $\#({C\cap f^{-n}(C)})$ indeed grows to infinity. Thus the pair $(f,C)$ falls into the framework of the Manin-Mumford problem. 
On the other hand  it is a  theorem by  E. Bedford and J. Smillie~\cite{BS1}   
that there exists no $f$-invariant algebraic curve.   
These examples motivate the following  conjecture.

\begin{conj}[Dynamical Manin-Mumford conjecture for complex polynomial automorphisms of the affine plane]
\label{conj:MM2}
Let $f$ be a complex polynomial automorphism of H\'enon type of the affine plane. Assume that there exists 
 an irreducible algebraic curve $C$
containing  infinitely many periodic points of $f$.

Then there exists an involution $\sigma$ of the affine plane whose set of fixed points is $C$ and an integer $n\ge 1$ 
such that  $\sigma f^n \sigma = f^{-n}$.  
\end{conj}

Recall that the Jacobian $\jac (f)$ of a polynomial automorphism is a non-zero constant.
If $f$ is reversible then $\jac (f)  = \pm 1$.  In particular, if $f^n$ is reversible for some $n$,  $\jac(f)$ must be a root of unity. 

Our first main result can thus be seen as a step towards Conjecture \ref{conj:MM2} 
(see Remark \ref{rem:symmetry} below for   comments about the asserted symmetry in the conjecture).

\begin{mthm}\label{thm:mainArch}
Let  $f$ be  a   polynomial automorphism of H\'enon type  of the affine plane, defined over a field of characteristic zero.
Assume that there exists an algebraic curve containing infinitely many periodic points of $f$. 

Then the Jacobian $\jac(f)$ is algebraic over $\Q$ and all its Galois conjugates have complex modulus $1$. 
\end{mthm}

In particular if $\jac(f)$ is an algebraic integer then it is a root of unity. 

Using a specialization argument, one can reduce the prof to the case $f$ and $C$ are both defined over a number field $\LL$
 (see \S \ref{sec:general fields}).
Fix an algebraic closure $\LL^{\rm alg}$ of $\LL$.    Modifying 
 a construction of S. Kawagu\-chi~\cite{kawaguchi06},  C.-G.~Lee~\cite{lee} built a dynamical 
height function $h_f : \A^2(\LL^{\rm alg}) \to \R_+$. This height  is associated to a continuous semi-positive adelic metrization of the ample line bundle $\mathcal{O}_{\P^2}(1)$ (in the sense of Zhang \cite{Z95}) and    $h_f(p) =0$ if and only if $p$ is periodic. 
We refer the reader to the survey~\cite{chambert11} for 
a detailed  account on 
these concepts.

When $f$ and $C$ are defined over a number field, 
Theorem~\ref{thm:mainArch}   is now a consequence of    the following  effective statement. 

 \begin{theoprime}\label{thm:maineffective}
Let  $f$ be  a polynomial automorphism of H\'enon type  of the   affine plane, defined over a number field $\mathbb{L}$. 
Assume that  there exists an archimedean place $v$ such that $|\jac(f)|_v \neq 1$.

Then for any algebraic curve $C$ defined over $\mathbb L$  there exists a positive constant $\varepsilon = \varepsilon(C) >0$ such that  the set $\{ p\in C(\LL^{\rm alg}), \, h_f(p) \le \varepsilon\}$ is  finite.
\end{theoprime}

Let us briefly explain the strategy of the proof. 
We follow the approach of L. Szpiro, E.  Ullmo and S.-W. Zhang~\cite{SUZ,ullmo,Z98} to 
the Bogomolov conjecture whose statement is the analog of Theorem~\ref{thm:maineffective}
in case $f$ is the doubling map on an abelian variety.

\medskip

The first step is to describe the asymptotic  distribution of the periodic points lying on $C$. 
Pick any place $v$ on $\LL$ and denote by $\LL_v$ the completion of the algebraic closure
of the completion of $\LL$ relative to the norm $v$.
Write $\|(x,y)\|_v = \max \{ |x|, |y| \}$ and $\log^+ = \max \{ \log , 0\}$. Then it can be shown
 that the sequence of functions $\frac1{d^n} \log^+ \| f^n(x,y) \|_v$ converges uniformly on bounded sets in $\LL_v^2$ to a 
 continuous ``Green" function $G^+_v: \LL_v^2 \to \R_+$ satisfying the invariance property 
 $G^+\! \circ f = d G^+$, where $d = \deg(f)$.
  Its zero locus $\{G^+ = 0\}$   coincides with the set of points in $\LL_v^2$ with bounded forward orbits. 

Replacing $f$ by its inverse, one defines a function $G^-$ in a similar way and we
set $G = \max ( G^+, G^-)$. These Green functions were 
first introduced and studied in the context of complex polynomial automorphisms by
J.~H.~Hubbard \cite{hu},
E. Bedford and J. Smillie~\cite{BS1} and J.E. Fornaess and N. Sibony~\cite{fornaess-sibony}. 

The key observation is that the asymptotic distribution of periodic points on $C$ can be understood by 
applying suitable equidistribution result for points of small height on curves. 
These results were developed by various authors in greater generality and the version we use here  is due to
P. Autissier~\cite{autissier} and A. Thuillier~\cite{thuillier-PhD}.
More precisely, we prove that the collection of functions $\{G^+_v\}_v$ (resp. $\{G^-_v\}_v$) induces continuous semi-positive  metrizations on $\mathcal{O}_C(1)$ at all places. Then the Autissier-Thuillier theorem
implies that the probability measures equidistributed over Galois conjugates of periodic points in $C$ converge to a multiple of 
$\Delta G^+_v|_C$ (resp. $\Delta G^-_v|_C$) at any place\footnote{At a non-Archimedean place,   $\Delta$ stands 
for the Laplacian operator as defined by Thuillier.}
when the period tends to infinity. 

From this one deduces that  for each $v$ the functions $G^+_v$ and $G^-_v$ are proportional on $C$ (up to a harmonic function).

\medskip

The second step is 
 to use this information on the Green functions to infer that $f$ is conservative at archimedean places.
 The argument relies on Pesin's theory and is quite technical so that  let us first explain how the mechanism   works
  under a more restrictive assumption. 

Suppose indeed that there exists a hyperbolic periodic point $p$ in the regular locus of $C$, with 
multipliers $u$, $s$ satisfying  $|u| > 1 > |s|$, and assume
moreover that the local unstable manifold $W^u_{\rm loc}(p)$, the local stable manifold $W^s_{\rm loc}(p)$, 
and the curve $C$ are pairwise transverse. Using the invariance property of $G^+$, 
we can  compute the local H\"older exponent $\vartheta_+$ of $G^+$ at $p$ along 
$W^u_{\rm loc}(p)$, which satisfies  the equality
$|u|^{\vartheta_+} = d$. Using a rescaling argument reminiscent of that used by X. Buff and A. Epstein 
in~\cite{buff-epstein}, we then show that this H\"older exponent is actually equal to that of $G^+|_C$. 
Applying the same argument to $f^{-1}$,
we get that  the local H\"older exponent $\vartheta_-$ of $G^-$ along the stable manifold satisfies $|s|^{-\vartheta_-} = d$.
But      since
$G^+|_C$ and $G^-|_C$ are proportional, $\vartheta_-$ and $\vartheta_+$ must be equal. This proves that $|\jac(f)| = | us | =1$.

\medskip

Unfortunately  we cannot ensure the existence of such a saddle point at an archime\-dean place. 
It turns out that working  at all  places (archimedean or not)  
resolves this difficulty. Adapting the above argument then leads to our next main result. 

\begin{mthm}\label{thm:mainNA}
Let $f$ be a polynomial automorphism of H\'enon type, defined over a field of characteristic zero. 
Assume that there exists an irreducible curve $C$ containing infinitely many periodic points of $f$. 
We suppose in addition that the following transversality statement is true: 
\begin{itemize} 
\item[(T)]  There exists a periodic point $p\in \mathrm{Reg}(C)$ such that 
$T_pC$  is not periodic under the induced action of  $f$.
\end{itemize}
Then  $\jac(f)$ is a root of unity. 
\end{mthm}

Observe that this result lies very much in the spirit of~\cite[Conjecture~2.4]{GTZ}. Let us  also note that if $C$ contains a saddle point at an archimedean place, then the transversality assumption (T) is superfluous (see Theorem \ref{thm:saddle archimedean}).

\medskip

Returning to the proof of Theorem~\ref{thm:mainArch} 
we get around the issue of the existence of a hyperbolic periodic point on $C$ 
and that   of the transversality 
of its invariant manifolds with $C$ 
by  applying Pesin's theory of non-uniform hyperbolicity, in combination with 
the theory of laminar currents, in the spirit of the work 
of E. Bedford, M. Lyubich and J. Smillie \cite{BLS93}. 
 This allows to  estimate the H\"older exponent of $G^+$ at   generic points
and relate it to the positive Lyapunov exponent of
the so-called {\em equilibrium measure} $\mu_f := (dd^c)^2 \max(G^+, G^- ) $. 
This is an ergodic invariant measure 
 which has remarkable properties; 
 in particular 
 it   describes the asymptotic distribution of periodic orbits 
 see~\cite{bls2}.
 
The proportionality of  $G^+$ and $G^-$  on $C$ finally implies that  the positive and negative exponents 
are opposite, thereby showing that $|\jac(f)| =1$.

The key input of Pesin's theory in our argument is to guarantee the \emph{transversality} of 
stable and unstable manifolds at a $\mu_f$-generic point with the curve $C$.

\medskip

A dual way to state Theorem~\ref{thm:mainArch} is to say that the intersection of the set of periodic points
with any curve is finite when $|\jac(f)| \neq 1$. We expect that the following stronger uniform statement holds.

\begin{conj} 
Let $f$ be a complex polynomial automorphism of H\'enon type such that $|\jac(f)| \neq 1$.
Then for any algebraic curve $C$,
 the cardinality of the set of periodic points of $f$ lying  on $C$ is bounded from above by a constant depending
 only on the degree  of $C$,  on the degree of $f$ and the Jacobian $\jac(f)$. 
\end{conj}

We indicate   in \S \ref{subs:uniformA'} how to adapt the arguments of Theorem \ref{thm:maineffective} to
confirm a weaker form of this conjecture.

The  automatic uniformity statement
obtained by T. Scanlon~\cite{aut-uniform}, based on ideas of E. Hrushovski, implies that such a bound 
would follow from   (a restricted version of) the dynamical Manin-Mumford problem for 
product maps of the form $(f, f, \ldots, f)$ acting on $(\A^2)^n$. Even though this  problem seems very delicate, 
we are able to address some cases of the dynamical Manin-Mumford problem for special 
product maps of H\'enon type. 

In the second part of the paper we    prove the following   two theorems. 

\begin{mthm}\label{thm:unlikely}
Let  $f$ and $g$ be two  polynomial automorphisms of H\'enon type 
of the affine plane defined over a number field.

If  $f$ and $g$ share a set of periodic points that is Zariski dense, then there exist two non-zero integers $n,m\in \Z$ 
such that $f^n = g^m$.
\end{mthm}

\begin{mthm}\label{thm:unlikelycomp}
Let  $f$ and $g$ be two polynomial automorphisms of H\'enon type
of the affine plane with complex coefficients such that $|\jac(f)| \neq 1$. 

If  $f$ and $g$ share an infinite set of periodic points, then there exist two non-zero integers $n,m\in \Z$ 
such that $f^n = g^m$.
\end{mthm}
 
Notice  that these two statements concern product maps $(f,g)$  such that the diagonal in  
$\A^2 \times \A^2$ admits  a Zariski dense   
set of periodic points \footnote{Observe that 
in Theorem \ref{thm:unlikelycomp}, this Zariski density follows from Theorem~\ref{thm:mainArch}.} 

We also show in \S\ref{sec:cycles} that Theorem~\ref{thm:unlikely} holds for a pair of 
automorphisms sharing infinitely many periodic {\em cycles}.

Observe that  Theorems~\ref{thm:unlikely} and~\ref{thm:unlikelycomp} are generalizations to our setting 
of  recent results due to M. Baker and L. DeMarco~\cite{baker-demarco} and 
X.~Yuan and S.-W.~Zhang~\cite{YZ13a,YZ13b}.

\medskip

We believe that these results hold under the following weaker assumption. 

\begin{conj}
Suppose $f$ and $g$ are two complex polynomial automorphisms of H\'enon type 
sharing  infinitely many periodic points. 
Then $f^n = g^m$ for some non-zero integers $n$ and $m$.
\end{conj}

The proof of Theorems~\ref{thm:unlikely} goes as follows. The hypothesis implies that the 
equidistribution theorem for points of small height (X. Yuan~\cite{yuan}, C.-G. Lee \cite{lee})
 can be applied. Therefore $f$ and $g$ have the same equilibrium measure. 
If it happens that  $f$ and $g$ are simultaneously  conjugate   to automorphisms 
that extends to birational maps on  $\P^2$ contracting the line at infinity to a point that is not indeterminate, 
then   it is not difficult to see that the Green functions of $f$ and $g$  coincide.
 We can then invoke a theorem of   S. Lamy~\cite{lamy} to conclude  that $f$ and $g$ have a common iterate. 
  
  Otherwise we use the equality of equilibrium 
  measures at all places to infer that $f$ and $g$, as well as any automorphism belonging to the group generated by $f$ and $g$,
    have the same sets of periodic points.  Then we use Lamy's geometric group theoretic description of
  $\aut[\A^2]$ to reduce the situation to the previous one.

Theorem~\ref{thm:unlikelycomp} 
is obtained from Theorem~\ref{thm:unlikely} by a specialization argument. 
  Theorem~\ref{thm:mainArch} is used in the course of the proof, 
  which explains the  need of an extra hypothesis on the Jacobian of one of the maps. 


\begin{center}$\diamond$\end{center}

\medskip

The plan of the paper is as follows. In \S \ref{sec:prel} we gather a number of  facts on the dynamics of polynomial automorphisms over arbitrary metrized fields, including equidistribution theorems for points of small height. Then in \S \ref{sec:equidistribution} 
we show how  these equidistribution results apply in our situation. 
 Theorems \ref{thm:mainArch},  \ref{thm:mainNA},  are respectively established in 
   \S \ref{sec:dissipative} and  \ref{sec:transversality}, assuming that $f$ and $C$ are defined over a number field. The extension to general ground fields is achieved in \S \ref{sec:general fields}. 
   The proofs of Theorems~\ref{thm:unlikely} and ~\ref{thm:unlikelycomp} are given in \S \ref{sec:infinite}. Finally,   \S \ref{sec:reversible} is devoted to a 
   discussion on  reversible mappings.

\medskip

\noindent {\bf Acknowledgment.} The authors would like to thank J.-L. Lin for  several discussions, D. Roessler for his explanations on the Manin-Mumford conjecture, S. Lamy for useful indications on reversible polynomial automorphisms, and A. Ducros for his help in proving  Lemma~\ref{lem:ducros}.

\section{Polynomial automorphisms over a metrized field}\label{sec:prel} 

We classically write $\log^+ $ for $ \max \{ \log , 0\}$. In the first four sections, 
we fix an arbitrary complete (non trivially) metrized field $(L, |\cdot |)$ of characteristic zero that is algebraically closed.
In \S\ref{subs:equidist} and \ref{subs:height}, we work over a number field $\LL$. 

\subsection{Potential theory over a non-Archimedean curve}\label{sec:pot-NA}
In this section we suppose that the norm on $L$ is non-Archimedean and  
give a brief account on  Thuillier's potential theory on curves \cite{thuillier-PhD}. 

\smallskip

Pick any smooth algebraic curve $C$ defined over $L$. We shall work with the analytification $C^{\an}$ of $C$ in the sense of Berkovich~\cite[\S 3.4]{berkovich}. If $U \subset C$ is a Zariski  affine open subset of $C$, then its analytification $U^{\an}$ is defined as the set of multiplicative seminorms on the ring of regular functions $L[U]$
whose restriction to $L$ equals $|\cdot|$, endowed with the topology of
 pointwise convergence. Any closed point $p \in C$ defines a point in $U^{\an}$
given by $  L[U] \ni \phi \mapsto |\phi(p)|\in\R_+$.

The space $C^{\an}$ is then constructed by patching together the sets $U^{\an}$
where $U$ ranges over any affine cover of $C$. In this way, one obtains 
a  locally compact and connected space. There is a distinguished set of compact subsets of $C^{\an}$ that forms a basis for its topology and are referred to as strictly $L$-affinoid subdomains.  We refer to \cite[\S 3]{berkovich} for a formal definition.  For us  it will be sufficient 
 to say that each affinoid subdomain $A$ has a finite boundary, and admits a canonical retraction to its skeleton
$\sk(A) \subset A$ which is the geometric realization of a finite graph.  We write $r_A: A \to \sk(A)$ for this retraction.

Any of these skeletons comes equipped with a canonical integral affine structure, hence with
a metric. One can thus make sense of the notion of a harmonic function on $\sk(A)$. 
By definition this is a continuous function that it is piecewise affine, and such that the sum of the directional derivatives at any point
 (including the endpoints) is zero. 

\smallskip

A {\em harmonic} function $ h : U\to \R$ defined on a (Berkovich) open subset $U\subset C^{\an}$ is  a continuous function such that for all subdomains $A$ the map $h|_{\sk(A)}$ is harmonic.

For any invertible function $\phi\in L[U]$ defined on an affine open subset $U\subset C$,
the function $\log |\phi|$ is harmonic on $U^{\an}$, see~\cite[Proposition 2.3.20]{thuillier-PhD}. However it is not true that any harmonic
function can be   locally expressed 
 as the logarithm of an invertible function, see~\cite[Lemme 2.3.22]{thuillier-PhD}. This discrepancy with the complex case will however not affect our arguments. 

\begin{prop}\label{prop:uniform-harmonic}
Pick any open subset $U$ of  $C^{\an}$, and 
suppose that $h_n$ is a sequence of harmonic functions defined on $U$, that converges 
uniformly.  Then its limit $\lim_n h_n$ is harmonic.
\end{prop}

\begin{proof}
The result is a consequence of the following fact: 
suppose we are given a sequence of convex functions on a real segment that converges 
locally uniformly.  Then the limit is convex and the directional derivatives also converge at any point.
\end{proof}

\begin{prop}\label{prop:maximum-principle}\cite[Proposition 2.3.13]{thuillier-PhD}
Suppose $u$ is a non-negative harmonic function, such that $h(p) =0$
for some point $p\in U$. 

Then $h$ is  constant in a neighborhood of $p$.
\end{prop}

Pick any connected open subset $U\subset C^{\an}$.
An upper-semicontinuous function $u: U\to \R\cup\{ - \infty\}$ is said to be {\em subharmonic} if it is not identically $-\infty$ and
 satisfies the condition that for any strictly $L$-affinoid subdomain $A$ and any harmonic function $h$ on $A$ then $u|_{\partial A} \le h|_{\partial A}$ implies $u \le h$ on $A$.

\smallskip

One can check that the set of subharmonic functions is a positive convex cone that is stable by taking maxima, 
 contains all functions of the form $\log |\phi|$ for any regular function $\phi$, and is stable under decreasing sequences\footnote{We shall be concerned \emph{only} with subharmonic functions that are uniform limits of positive linear combinations of maxima of functions of the form $\log|\phi|$.}, see~\cite[Proposition 3.1.9]{thuillier-PhD}.

\medskip

To any subharmonic function $u$ defined on an open set $U\subset C^{\an}$ is associated a unique positive Radon measure $\Delta u$ supported on $U$ that satisfies the following properties:
\begin{itemize}
\item 
$\Delta ( au + v) = a \Delta u + \Delta v$ for any two subharmonic functions $u,v$ and any positive constant $a>0$; 
\item
for any regular function $\phi$, the Poincar\'e-Lelong formula holds:
$$
\Delta \log |\phi| = \sum_{\phi(p)  =0} \ord_p(\phi) \delta_p~; 
$$
\item
for any decreasing sequence $u_n \to u$,   $\Delta u_n$ converges to $\Delta u$ in the weak sense of measures.
\end{itemize}

We shall use the following properties of this Laplacian operator.
\begin{prop}\label{prop:laplacian-char-harm}
Let $u: U \to \R \cup \{ - \infty\}$
be any subharmonic function. 

Then $u$ is harmonic iff $\Delta u =0$.
\end{prop}

\begin{proof}
If $u$ is harmonic then $\pm u$ are subharmonic as in \cite[Definition 3.1.5]{thuillier-PhD}, hence
$\pm \Delta u$ is a positive measure by \cite[Th\'eor\`eme 3.4.8]{thuillier-PhD}, and $\Delta u =0$.

Conversely, if $\Delta u =0$ then $u$ is harmonic by \cite[Corollaire 3.4.9]{thuillier-PhD}.
\end{proof}

\begin{prop}\label{prop:laplacian-invariance}
Suppose $U,V$ are two  respective 
open subsets of the Berko\-vich analytification of two smooth algebraic curves, and  $f : U  \to V$ is an isomorphism. Let $u$ be any subharmonic function on $V$.

Then $u \circ f$ is subharmonic on $U$, and $\Delta (u \circ f)  = f^* \Delta u$.
\end{prop}

\begin{proof}
The first statement follows from  \cite[Proposition 3.1.13]{thuillier-PhD}, and the second from
\cite[Proposition 3.2.13]{thuillier-PhD}.
\end{proof}

\subsection{Dynamics of regular automorphisms}
Following Sibony \cite{sib-survey} we say that a   polynomial automorphism of the affine plane 
$f: \A^2 \to \A^2$ is   {\em regular} if its extension
as a rational map to the projective plane $F: \P^2 \dto \P^2$ 
contracts the line at infinity $H_\infty$ to a point $p_+$
that is \emph{not} indeterminate for $F$. 
It follows that $p_+$ is a super-attracting fixed point, and that its inverse map contracts $H_\infty$ to the (single)
point of indeterminacy $p_-$ of $F$.

The degree of a polynomial map is the maximum of the degrees of its (two) components. 
By \cite{friedland-milnor}, up to a linear change of coordinates,  any regular polynomial automorphism of degree $\ge2$ is the composition of finitely many maps of the form
$$
(x,y) \mapsto (ay, x + P(y))
$$
where $a\in L^*$, and $P$ is a polynomial of degree $\ge2$.

A  polynomial automorphism of $\A^2$ will be said to be of H\'enon type if it is conjugated (in the group of automorphisms) to a regular automorphism of degree $\ge 2$. A  complex polynomial automorphism has positive topological entropy if and only if it 
is of H\'enon-type.

\medskip

Let $f : \A^2 \to \A^2$ be any regular polynomial automorphism of degree $d \ge 2$
and such that  $p_+ = [0:1:0]$ and $p_- = [1:0:0]$ in  homogeneous coordinates on $\P^2$.

Fix a  constant $C>0$, and  define
\begin{eqnarray*}
V\,\, &=& \{p= (x,y) \in L^2, \, \|p\| = \max \{ |x|, |y|\} \le C\}, \\ 
V^{+} &=& \{ (x,y) \in L^2, \,   |y| \ge \max\{ |x| , C\} \},\\ 
V^{-}  &=& \{ (x,y) \in L^2, \, |x| \ge \max\{ |y| , C\} \}.
\end{eqnarray*}
It can be shown that if $C$ was chosen  sufficiently large, then $f(V^+)\subset V^+$, and more precisely
\begin{equation} \label{eq:clas}
\frac1d \log | y \circ f | \ge \log |y| - \cst
\end{equation}
for any point in $V^+$. The same kind of inequaliy holds when $f$ is replaced by its inverse, so one obtains that
$f^{-1} (V^-) \subset V^-$. This implies
 $f(V) \subset V \cup V^+$.

Let us set 
\begin{eqnarray*}
K&=& \{p \in L^2, \, \sup_{n \in \Z} \|f^n(p)\| < +\infty \}~\text{ and} \\ 
 K^\pm&=& \{p \in L^2, \, \sup_{n \ge 0} \|f^{\pm n}(p)\| < +\infty \}~.
\end{eqnarray*}
The next result easily  follows from the above properties.
\begin{lem}
We have\begin{itemize}
\item $K = K^+ \cap K^- \subset V$; 
\item $L^2 \setminus K^+ = \bigcup_{n\ge 0} f^{-n} (V^+)$;
and  $K^+\subset V \cup V^-$; 
\item $L^2 \setminus K^- = \bigcup_{n\ge 0} f^{-n} (V^-)$;
and  $K^-\subset V \cup V^+$.
\end{itemize}
\end{lem}

It follows from~\eqref{eq:clas} that the sequence $\frac1{d^n} \log^+ \| f^n(p) \|$ converges uniformly on 
$L^2$ to a  non-negative continuous function that we denote by $G^+$. 
Similarly  one can define $G^-(p) = \lim_{n\to\infty}\frac1{d^n} \log^+ \| f^{-n}(p) \|$.
The next result collects some properties of  these functions, see~\cite{BS1} and \cite[Theorem~A]{kawaguchi13}. 

\begin{prop}\label{prop:Gplus}
The functions $G^+, G^-$ are continuous non-negative functions on $L^2$ which  satisfy
\begin{enumerate}
\item
$G^\pm \circ f^{\pm 1} =  d G^{\pm}$;
\item
$G^\pm (p) - \log^+ \| p\|$ extend to  continuous functions to $\P^2(L) \setminus \{ p_\mp\}$ that are bounded from above;
\item 
$\{G^\pm =0\} = K^\pm$.
\end{enumerate}
\end{prop}

 The following fact will be crucial in our work. It follows from~\cite[Proposition~4.2]{BS1} whose proof works over any field. 
 
 \begin{prop}\label{prop:no fixed curves}
A polynomial automorphism of H\'enon type admits no invariant algebraic curve. 
 \end{prop}

\subsection{Invariant measures}\label{sec:inv-meas}
We keep  notation as in the previous subsection. 
Our purpose is to construct an invariant measure from the functions $G^+$ and $G^-$. 

The next result follows directly from Proposition~\ref{prop:Gplus} above.

\begin{prop}\label{prop:continuity of G}
The function $G = \max \{ G^+, G^-\}$ is a  continuous non-negative function on $L^2$ such that:
\begin{enumerate}
\item
$G (p) - \log ^+\| p\|$ extends to a continuous function to $\P^2(L)$;
\item 
$\{G =0\} = K$.
\end{enumerate}
\end{prop}

Assume first that $(L,|\cdot|)$ is archimedean, i.e.   $L = \C$ endowed with its standard hermitian norm. In this case, $G^+$ and $G^-$ are continuous plurisubharmonic functions on $\C^2$ and so is $G$. Using Bedford-Taylor's theory it is possible to make sense of the 
Monge-Amp\`ere of $G$ and define the positive measure $\mu_f:= (dd^c)^2 G$. 
It is a    $f$-invariant probability measure whose support is included in $K$.
We refer to~\cite{BLS93} for more details on its ergodic properties. 

Pick any irreducible algebraic curve $C$ and denote by $\reg(C)$ its set of regular points.  
Then  $\mu_{f,C}$ is by definition 
the Laplacian of the function $G$ restricted to $\reg(C)$. Since $G$ is continuous and $G (p) - \log^+ \| p\|$ is bounded,  this measure carries  no mass on points and its mass equals $\deg(C)$.

\smallskip

When $(L,|\cdot|)$ is non-Archimedean,  the analogues of the measures $\mu_f$ and $\mu_{f,C}$ have
 been constructed by Chambert-Loir~\cite{chambert06,chambert11}. 

Indeed the function $G$ induces a metrization $|\cdot|_G$ on the line bundle $\mathcal{O}(1)_{\P^2}$ by setting $|\sigma|_G := \exp(-G)$, where $\sigma$ is the section  corresponding to the constant function $1$ on $\A^2$.

Proposition~\ref{prop:continuity of G} together with  the fact that $G^+$ and $G^-$ are uniform limits of multiples of functions of the form $\log \max \{ |P_1|, |P_2|\}$ with $P_i \in L[x,y]$
imply that the metrization $|\cdot|_G$ is a continuous semi-positive metric in the sense of~\cite[\S 3.1]{chambert11}.

The measure  $\mu_f$ is defined as a probability measure 
on the Berkovich analytic space $\A^{2, {\rm an}}_L$. 
This measure is $f$-invariant, however the study of its ergodic properties remains
to be done.

When the affine plane is replaced by an irreducible curve $C$, the measure $\mu_{f,C}$ is 
then a positive measure on the analytification $C^{\rm an}$ of $C$ in the sense of Berkovich.
It can be 
 defined using Thuillier's theory recalled in \S\ref{sec:pot-NA} as $\mu_{f,C} := \Delta G |_C$ . Its mass is again equal to $\deg(C)$.

\subsection{Saddle fixed points}\label{ss:saddle}
 In this subsection we let 
$f$ be any analytic germ fixing the point $0 \in \A^2_L$. 
The fixed point $0$ is said to be a \emph{saddle} when 
 the eigenvalues $u,s$ of $Df(0)$
satisfy $|u| > 1 > |s|$.

\smallskip

Given a small bidisk $B$ around 0, we let 
 $W^s_{\rm loc}(0)$ (resp. $W^u_{\rm loc}(0)$) to be the set of points $p\in B$ such that
 for every $n\geq 0$, $f^n(p) \in B$ (resp. $f^{-n}(p) \in B$). It follows that if $p\in  W^s_{\rm loc}(0)$ (resp. $p\in W^u_{\rm loc}(0)$)
$\lim_{n\to \infty} f^n(p)= 0$ ($\lim_{n\to \infty} f^{-n}(p) = 0$). 
It is a theorem that $W^s_{\rm loc}$ and $W^u_{\rm loc}$ are graphs of analytic functions in a neighborhood of $0$ tangent to the respective eigendirections of $df$, hence  intersect 
transversely. 
We refer to~\cite[Theorem~A.1]{HY83} for a proof that works over any metrized field. We refer to $W^s_{\rm loc}(0)$ (resp. $W^u_
{\rm loc}(0)$) as the local stable (resp. unstable)  manifold (or curve) of $0$.

It is easy  that one may always make a change of coordinates  such that 
$W^u_{\rm loc}(0) = \set{y=0}$ and $W^s_{\rm loc}(0) = \set{x=0}$. 
We  will need the following more precise normal form.

\begin{lem}\label{lem:coordinates}
There exist coordinates $(x,y)$ near $0$ in which $f$ assumes the form 
 $$f(x,y) =\left( ux(1+xy g_1(x,y)), sy(1+xy g_2(x,y)) \right)~,$$
where $g_1, g_2$ are analytic functions.
 \end{lem}
 
 Note that   in this set of coordinates  $f$ is linear along the stable and unstable manifolds.  
  
 \begin{proof}
By straightening the local stable and  unstable manifolds, $f$ can be put under the form  
  $$
f(x,y) = (ux (1+ \hot), s y(1+ \hot))  \text{, with } \abs{u}>1 \text{ and }    \abs{s}<1.
$$
and we want to make this expression more precise. First, by the Schr\"oder linearization theorem (which holds for arbitrary $L$, 
see~\cite{HY83}) we can make a local change of coordinates depending only on $x$ in which 
$f\rest{W^u_{\rm loc}(0)}$ becomes linear. 
Doing the same for $y$, we reach the form
\begin{equation}\label{eq:us2}
 f(x,y) = (ux(1+  y g^{(0)}(x,y)) , s y (1+x  h^{(0)}(x,y)).   
\end{equation}
 
Let us now focus on the first coordinate in \eqref{eq:us2}. We want to get rid of monomials of the form $xy^j$ for $j>0$. 
 Re-order the expression of $f$ so that it writes as
$$f (x,y) = (u x a_1(y) + x^2 a_2(y) + \cdots   , sy b_1(x) + y^2b_2(x) + \cdots ) , $$ where the $a_j$ and $b_j$ are analytic and $a_1(0) = b_1(0) = 1$. We want to find local coordinates in which $a_1(y) \equiv 1$. 
For this, put 
$(x',y') = (\varphi(y) x,y)$, with $\varphi(0 )=1$.  Notice that in the   coordinates $(x',y')$, 
 $f$ preserves the coordinate axes and remains linear  along them, so $f$ is still of the form \eqref{eq:us2}. 
 In the new coordinates, $f$ expresses as 
$$(x',y') \mapsto \lrpar{ u a_1(y') \frac{\varphi(sy')}{\varphi(y')} x' + O((x')^2), sy'  + O(x')  },$$ thus to achieve 
$a_1(y') \frac{\varphi(sy')}{\varphi(y')}=1$ it is enough to choose $\varphi(y')  = \prod_{n=0}^{\infty} a_1 (s^ny')$, which is well-defined for sufficiently small $y'$,  since $\abs{s}<1$ and $a_1(0) = 1$. Doing the same in the second variable, and renaming the coordinates  
as $(x,y)$, we obtain 
$$f (x,y) = (u x   + x^2 a_2(y) + \cdots   , sy   + y^2b_2(x) + \cdots ).$$ Going back to the form \eqref{eq:us2}, we  get the desired result.
 \end{proof}

\subsection{Stable manifolds of polynomial automorphisms}
In the case of polynomial automorphisms local stable manifolds can be globalized  and have the following structure. 

\begin{prop}\label{prop:stablecurve}
Let $f : \A^2 \to \A^2$ be any polynomial automorphism of  $\A^2$
and assume that  $0$ is a saddle fixed point for $f$. Denote by $s$ the eigenvalue of $Df(0)$ lying in the unit disk.

Then the global stable manifold $W^s(0) := \{p \in L^2, \, f^n(p) \to 0\}$ is an immersed affine line.  More precisely,  there exists an analytic injective immersion $\phi_s : \A^1 \to \A^2$ whose image is $W^s(0)$, and such that $f \circ \phi_s (\zeta) = \phi_s( s\zeta)$ for every 
$\zeta \in \mathbb A^1$.
 \end{prop}

\begin{proof}
We saw in \S\ref{ss:saddle} that for a sufficiently small neighborhood $N$ of $0$ the intersection $W^s_{\rm loc} (0) = W^s(0) \cap N$ is 
parameterized by an analytic immersion $\phi_s : B(0,1) \to N$ which satisfies 
\begin{equation}\label{eq:invmfd}
f \circ \phi_s (\zeta) = \phi_s( s\zeta)~.
\end{equation}
Since $f$ is an 
automorphism, it follows that $W^s(0) = \bigcup_{n\ge 0} W^s_{\rm loc}(0)$, and  using the   functional equation 
we may extend the analytic immersion $\phi_s$    by putting 
$$ 
\phi_s (\zeta) = f^{-n}\phi_s( s^n\zeta)
$$
for every $\zeta \in L$ and sufficiently large $n$.
\end{proof}

\begin{prop}\label{prop:stable and G-}
Let $f : \A^2 \to \A^2$ be a  polynomial automorphism of H\'enon type of  $\A^2$
and assume that $0$ is a saddle fixed point for $f$.

Then the restriction of $G^-$ to $W^s(0)$ is not  identically $0$. In particular $G^-\rest{W^s(0)}$ cannot vanish identically  in a neighborhood of the origin. The same results hold for $G^+\rest{W^u(0)}$. 
\end{prop}
 
\begin{proof}
Suppose by contradiction that $G^- \equiv 0 $ on $W^s (0)$.
By Proposition~\ref{prop:Gplus} (3), we have $W^s(0) \subset K^-$.
Since the successive images by $f$ of any point in  $W^s(0)$ converges to $0$ it follows that
$W^s(0) \subset K^+$ hence $W^s(0) \subset K$. 

We conclude that the image of the analytic map $\phi_s : \A^1 \to \A^2$ lies in a bounded domain hence it is a constant by Liouville's theorem (see~\cite{robert} for  the non-Archimedean case). This contradicts the fact that $\phi_s$ is an immersion. 

The second statement follows from the invariance relation~\eqref{eq:invmfd}.
\end{proof}

\subsection{Heights associated to adelic metrics on the affine space}\label{subs:equidist}
We now assume that $\LL$ is a number field, and we fix an algebraic closure $\LL^{\mathrm{alg}}$ of $\LL$. Let $\cM_\LL$ be the set of places of $\LL$, that is, of its multiplicative norms modulo equivalence. For each place $v\in \cM_\LL$ we denote by $|\cdot|_v$ the unique representative that is normalized in such a way that its restriction to $\Q$
is either the standard archimedean norm or the $p$-adic norm satisfying $|p|= p^{-1}$ for some prime $p >1$.

Then for any $x\in \LL$, the 
 {\em  product formula}  $\prod_{v\in \cM_\LL} |x|_v^{n_v} = 1$ holds. 
 Here the integer $n_v$ is the degree of the field extension of the completion of $\LL$ over the completion of $\Q$ relative to $|\cdot|_v$.

\smallskip

Fix an integer $d\ge1$.
We let $\LL_v$ be  the completion of the algebraic closure of the completion of $\LL$ relative to the absolute value $\abs{\cdot}_v$. For any $p = (x_1, \ldots, x_d) \in \LL_v^d$ we shall write
$\| p \|_v = \max \{ |x_1|, \ldots, |x_d|\}$ (or simply $\norm{p}$ when there is no risk of confusion).

Recall that one  the standard height of a point  $p\in \A^d(\LL^{\rm alg})$ is defined by the formula
\begin{eqnarray*}
h(p)=\frac{1}{\deg(p)}\sum_{v\in \cM_\LL}\sum_{q\in O(p)} n_v \log^+\|q \|_v 
\end{eqnarray*}
where $O(p)$ denotes  the orbit of $p$ under the absolute Galois group of $\LL$ and $\deg(p)$ is the cardinality of $O(p)$.

\smallskip

As in~\cite{chambert11} we use heights that are associated to semi-positive adelic metrics on ample line bundles. 
We present here this notion in a form that is taylored to our needs.

Suppose $X\subset \A^d_\LL$ is an irreducible affine variety. 
For us, $X$ will always  be either a curve or $\A^2_\LL$. 

A semi-positive adelic metric on the (ample) line bundle $\O(1)_X$ is a collection of functions 
$\{ G_v\}_{v\in \cM_\LL}$, $G_v : X(\LL^{\rm alg}) \to \R$ such that 
\begin{enumerate}
\item[(M1)] 
the function $G_v (p) - \log^+\| p\|$ extends continuously to the closure of $X$ in $\P^d(\LL)$ for each place $v$;
\item[(M2)] 
for all but finitely many places $G_v (p) = \log^+\| p\|$;
\item[(M3)] 
for each archimedean place the function $G_v$ is plurisubharmonic;
\item[(M4)] 
for each non-Archimedean place, the function $G_v$ is a uniform limit of positive multiples of functions of the form $\log \max \{ |P_1|, \ldots, |P_r|\}$ with $P_i \in \LL[x_1, \ldots, x_d]$.
\end{enumerate}

To any such semi-positive adelic metric $\{ G_v\}_{v\in \cM_\LL}$ is associated a height defined on $X(\LL^{\rm alg})$ by 
\begin{eqnarray*}
h_G(p)=\frac{1}{\deg(p)}\sum_{v\in \cM_\LL}\sum_{q\in O(p)} n_v G_{v} (q)~,
\end{eqnarray*}
and such that $\sup_{X(\LL^{\rm alg})} |h_G - h | < \infty$.

\smallskip

For any place $v$, one can also associate to the metrization $G_v$ 
a positive measure  $\MA (G_v)$ which is  defined in the same way as in \S\ref{sec:inv-meas}. 

When $v$ is archimedean then $\MA (G_v)$ is the Monge-Amp\`ere measure  
defined using Bedford-Taylor's theory whose mass is $\deg(X)$. 

When $v$ is non-Archimedean, then the measure $\MA (G_v)$ is defined by Chambert-Loir~\cite{chambert06} as a positive measure of mass $\deg(X)$ on the analytification of $X$ over $\LL_v$ in the sense of Berkovich. 
Its definition relies in an essential way on the condition (M4) above. 

When $X$ is a curve then $\MA(G_v)$ is alternatively defined as the Laplacian of $G_v|_C$
(in the sense of Thuillier when $v$ is a non-Archimedean place).

\subsection{Metrizations associated to polynomial automorphisms of H\'e\-non type and equidistribution}\label{subs:height}
Assume that $f$ is a regular polynomial automorphism of degree $\ge2$   defined over a number field $\LL$.

The following two results   are direct consequences of the definitions and Propositions~\ref{prop:Gplus} and~\ref{prop:continuity of G}.

\begin{prop}\label{pro:adelicG}
For any regular polynomial automorphism $f$ of degree $\ge2$
the collection $\{G_{v,f}\}$ defines a semi-positive adelic metric on $\mathcal{O}(1)_{\P^2}$.
\end{prop} 

In the sequel we denote by $h_f$ the height associated to this collection of metrics. 

\begin{prop}\label{pro:adelicCurve}
Let $f$ be a  regular polynomial automorphism  of degree $d\ge2$.  As above, denote by $p_+$
the fixed point at infinity of its rational extension to $\P^2$. 

Then for any irreducible algebraic curve $C$ whose Zariski closure $\bar{C}$ in $\P^2$ intersects the line at infinity at the single point $p_+$,   
the collection $\{G^+_{v,f}|_C\}$ defines a semi-positive adelic metric on $\O(1)_{\bar{C}}$.
\end{prop}

\smallskip

We will need two versions of the equidistribution theorem for points of small height: one is in restriction to a curve, 
and the other is in $\A^2$. 

The first statement follows from a statement due to 
P.~Autissier~\cite[Prop.~4.7.1]{autissier} at the archimedean place and to  A.~Thuillier~\cite[Th\'eor\`eme~4.3.6]{thuillier-PhD}.

\begin{thm}[Equidistribution for points of small height on a curve]\label{thm:yuan-curve}
Let $f$ be  a polynomial automorphism of H\'enon type defined over $\mathbb L$.
Suppose $C$ is an irreducible curve of the affine space $\A^2$ that is defined over $\mathbb L$  whose Zariski closure in $\P^2$ intersects the line at infinity only at $p_+$.

Suppose that we are given a infinite sequence of distinct points $p_m\in C(\LL^{\rm alg})$ such that $h_{G^+}(p_m) \to 0$.

Then, for any place $v\in \cM_\LL$,  the convergence 
\begin{equation}\label{eq:yuan1}
\frac{1}{\deg (p_m)}\sum_{q\in O(p_m)}\delta_q \longrightarrow \frac1{\deg(C)} \, 
\Delta(G^+_v|_{C_v})
\end{equation}
holds in the weak topology of measures, where $O(x_m)$ is the orbit of $x_m$ under the action of the absolute Galois group of $\LL$.
\end{thm}
For simplicity, we write  $C_v$ for the analytification of $C$ over the field $\LL_v$.

\begin{proof}
To keep the argument as short as possible we directly apply Yuan's result~\cite[Theorem~3.1]{yuan}. To do so one needs to check that the height of $C$ induced by the metrization given by $\{ G^+_v\}_v$ on $\mathcal{O}(1)_C$ is equal to zero. 
We refer to~\cite{chambert11} for a definition of this quantity. Now for a curve it follows from e.g.~\cite[Theorem~1.10]{Z95} that  this height is equal to 
$$
e = \mathrm{ess.} \inf_C h_{G^+} : = \sup_{\# F < \infty } \inf_{p \in C \setminus F} h_{G^+}(p)~.
$$
Our assumption implies that  $\inf_{p \in C \setminus F} h_{G^+}(p) \le \liminf_n h_{G^+}(p_n) = 0$.
On the other hand we have $G^+_v \ge0$ at all places hence $h_{G^+}(p) \ge0$ for every
$p \in C$. Therefore $e =0$ as required.
\end{proof}

The next result, still based on  Yuan's Theorem~\cite{yuan}, is due to C.-G.~Lee~\cite[Theorem~A]{lee}.

\begin{thm}[Equidistribution theorem for periodic points of H\'enon maps]\label{thm:yuan-henon}
Let $f$ be an automorphism of H\'enon type defined over $\mathbb L$.

Let $(p_m)_{m\geq 0}$ be any sequence  of distinct periodic points such that the set 
$\{ p_m\} \cap C$ is   finite   for any irreducible curve
$C\subset \A^2_\LL$.
Then, for any place $v\in \cM_\LL$,  the convergence 
\begin{equation}\label{eq:yuan2}
\frac{1}{\deg (p_m)}\sum_{q\in O(p_m)}\delta_q \longrightarrow 
\MA(G_v)
\end{equation}
holds in the weak topology of measures, where $O(x_m)$ is the orbit of $x_m$ under the action of the absolute Galois group of $\LL$.
\end{thm}

\section{Applying the equidistribution theorem}\label{sec:equidistribution}

A key step   of the proofs  of Theorems~\ref{thm:mainArch} and~\ref{thm:mainNA} 
 is  the use of equidistribution theorems  for points of small height. Doing so
we follow the approach to the Manin-Mumford conjecture initiated by 
Szpiro-Ullmo-Zhang~\cite{SUZ}. In our setting this results in the following proposition. 

\begin{prop}\label{prop:main2}
Let $f$ be a regular polynomial automorphism, and $C$ be an irreducible algebraic curve in the affine plane,
 both defined over a number field $\LL$.  
Suppose there exists a sequence of distinct points $p_n\in C(\LL^{\rm alg})$ 
%
such that 
$h_f(p_n) \to 0$.

Then for any place $v\in \cM_\LL$
there exists a positive rational number $\alpha:= \alpha(C,f)$ and a harmonic function $H_v : C_v \to \R$ such that 
$$
G^+_{v,f} =  \alpha\, G^-_{v,f} + H_v \text{ on } C.
$$

\end{prop}

Let us stress that  the constant $\alpha$ does not depend on the chosen place $v$, but only
on the curve $C$ and the automorphism $f$. When $v$ is non-Archimedean, the harmonicity of $H_v$ is understood in the sense of Thuillier, see \S \ref{sec:pot-NA}.

\begin{rem}
When $C$ has a single place at infinity\footnote{that is,
$C$ intersects the line at infinity in $\P^2$ at a single point and is analytically irreducible there.}, 
then it may be shown that $H_v$ must be a constant, hence necessarily zero since it vanishes at any periodic point of $f$. Taking $H_v \equiv 0$ somewhat simplifies  the proof of Theorem~\ref{thm:mainArch}.  
However it seems delicate to prove the vanishing of $H_v$ in the general case of a curve with several places at infinity. 
\end{rem}

\begin{proof}
Let $p_n$ be a sequence of distinct points in $C(\LL^{\rm alg})$ with $h_f(p_n) \to 0$, and fix any place $v\in \cM_\LL$.
To simplify notation we 
denote by $[F_n]$ the normalized equidistributed integration measure on 
the Galois orbit of $p_n$. It is a probability measure supported on the analytification $C_v$ of $C$
over $\LL_v$ in the sense of Berkovich.

By~\cite[Proposition~4.2]{BS1} 
together with Proposition \ref{prop:no fixed curves}, there exists an integer $k\ge0$ such that 
$f^k(C)$ intersects the line at infinity in $\P^2_\LL$ only at the super-attracting point $p_+$.
 
By Proposition~\ref{pro:adelicCurve}, the metrization given by $\{G^+_{v,f}\}$ is    semi-positive adelic. Let  $h_{G^+}$ be the  associated height. Since $$0 \le G_{f,v}^+ \le G_{f,v} \le d^k G_{f,v} \circ f^{-k}$$ at all places, it follows that 
 $$h_{G^+}(f^k(p_n)) \le h_f (f^k(p_n)) \le  d^k h_f (p_n) \to 0$$ whence $h_{G^+}(p_n)  \to 0$.
Theorem~\ref{thm:yuan-curve} thus applies
and we get that the sequence of probability measures $f^k_*[F_n]$
converges to the unique probability measure $\mu_{k}$ that is proportional to $\Delta(G^+_v |_{f^k(C_v)})$, that is 
   $$\mu_{k}=\frac1{\deg(f^k(C))} \, \Delta \left( G^+_{f,v}|_{f^k(C_v)}\right)~.$$
Pulling-back the convergence $f^k_*[F_n]\to \mu_{k}$ by the automorphism $f^k$, we get that 
\begin{align*}
\lim_n [F_n] &=
\frac1{\deg(f^k(C))} \, (f^k)^* \Delta \left(G^+_{f,v}|_{f^k(C_v)}\right) 
 \\
&= \frac1{\deg(f^k(C))} \,    \Delta\left( G^+_{f,v} \circ f^k \rest{C_v} \right) 
= 
\frac{d^k}{\deg(f^k(C))} \,   \Delta \lrpar{G^+_{f,v}|_{C_v}}~.
\end{align*}
Proceeding in the same way for $f^{-1}$ deduce that there exist  two non-negative integers
$k, k'$ such that 
$$
\frac{d^k}{\deg(f^k(C))} \,  \left.  \Delta \lrpar{ G^+_{f,v}\right|_{C_v}}=
\frac{d^{k'}}{\deg(f^{-k'}(C))} \, \left.  \Delta \lrpar{ G^-_{f,v}\right|_{C_v}} ~.
$$
Therefore, there exists a positive rational number $\alpha_C$  depending only on $C$ and $f$ such that 
 the restriction of the function  $G^+_{f,v} - \alpha_C G^-_{f,v}$ to $C_v$ is harmonic. The proof is complete.
\end{proof}

For completeness, let us  mention the following partial  converse to Proposition~\ref{prop:main2}. 
\begin{prop}
Suppose that there exists a positive constant $\alpha>0$,   such that 
$G^+_{v,f} =  \alpha \, G^-_{v,f}$ on $C$ for each place $v$.

Then there exists a sequence of points $p_n \in C(\LL^{\rm alg})$ such that
$h_f(p_n) \to 0$.
\end{prop}

\begin{rem} 
If for every place $v$ there exists a constant $\alpha_v$ such that  $G^+_{v,f}\rest{C} = \alpha_v G^-_{v,f}\rest{C}$ 
 then $\alpha_v$ does not depend on $v$.   
 This follows from the fact that the mass of $\Delta G^\pm_v |_{C_v}$  
only depends on the geometry of the branches of $C$ at infinity. 
\end{rem}

\begin{proof}
Replacing $C$ by $f^n(C)$ for $n$ large enough, one may suppose that $\alpha >1$ and that the completion of $C$ intersects the line at infinity only at the point $p_+$.  We claim that the height of $C$ is zero 
(see~\cite{chambert11} for a definition of the height of a curve). Indeed, 
$$
h_f(C) = \sum_v \int_{C_v} G_v \Delta G_v|_{C_v} + h_f(p_+) = \sum_v \int_{C_v} G^+_v \Delta G^+_v|_{C_v} =0,
$$ since $G^+_v\equiv 0$ on $\supp( \Delta G^+_v|_{C_v})$.
The fact that $h_f(p_+)=0$ can be obtained from~\cite[Theorem 6.5]{lee} which asserts that  
$$h_f (p) = \lim_n \frac1{\deg(f)^n} h_{\rm naive} (\phi_n(p))$$ where
$\phi_n : \P^2 \to \P^4$ is the regular map whose restriction to $\A^2$ is defined by $\phi_n(p) = (f^n(p), f^{-n}(p))$
and $h_{\rm naive}$ is the naive height on $\P^4$ (see \cite[\S B.2]{hindry silverman}).
 An easy computation shows that $\phi_n(p_+)$ is independent of $n$ so the result follows.

We then conclude by applying the arithmetic Hilbert-Samuel theorem, see~\cite[Theorem~1.10]{Z95}.
\end{proof}


\section{The DMM statement in the Archimedean dissipative case}\label{sec:dissipative}

Throughout this section we   assume that $f$ is a regular polynomial automorphism of 
$\mathbb{A}^2$ defined over a number field $\LL$. We use  the   notation and results from  Section~\ref{sec:prel}. 
Our purpose is to establish Theorem~\ref{thm:maineffective}. This in turn 
clearly implies Theorem~\ref{thm:mainArch} in the case where $f$ and $C$ are defined over a number field 
(the general case will be treated in \S \ref{sec:general fields} below). 

The proof  will be  based on Pesin's theory of non-uniformly hyperbolic dynamical systems. We refer to \cite{BLS93} for a presentation adapted to our situation.

\subsection{Proof of Theorem \ref{thm:maineffective}}\label{subs:number field}
Recall that $h_f$ denotes the height associated to the semi-positive adelic metric $\{G_{f,v}\}$.
We suppose that there exists an irreducible curve $C$ defined over $\LL$ and a sequence of points $p_n\in C(\LL^{\rm alg})$ with $h_f(p_n) \to 0$.

\smallskip

 We want to prove that $|\jac(f)|_v =1$.
To simplify notation we assume that  $\LL \subset \C$, 
drop the reference to $v$ and 
work directly over $\C$ endowed with its standard absolute value. 

Under our assumptions,  we know from Proposition~\ref{prop:main2}  that there exists a positive constant $\alpha$ such that 
$(G^+-\alpha G^-)\rest{C}$ is harmonic,
which implies that the positive measures 
$\mu^\pm_C := dd^c ( G^\pm |_C )$ are proportional.

Denote by $\chi^{u}$ and $\chi^s$    the Lyapunov exponents of $f$ relative to 
 the  measure $\mu = (dd^c)^2  G_f $. Recall that they are defined by
$$
\chi^u = \lim_{n\cv+\infty}  \int \log \; \norm{D f^n_p} d\mu(p) \text{ and }  \chi^s = \lim_{n\cv+\infty}  \int \log \; \norm{D f^{-n}_p} d\mu(p).
$$
We also  put $\lambda^{u} = \exp(\chi^{u})$ and $\lambda^s = \exp(-\chi^s)$. 
It is known that $\lambda^{u/s} \geq d >1$, and 
$\lambda^u  \lambda^s  = |\jac (f)|$, see~\cite{bs3}. 
The main step of the proof of Theorem \ref{thm:maineffective} is the following proposition,
which computes the lower H\"older exponent of continuity of $G^+$ at a $\mu^+_C$-generic point 
(observe that for such a point one has $G^+(p)=0$).

\begin{prop}\label{prop:pesin}
For $\mu^+_C$-almost every point $p$ in $C$, one has 
$$
\liminf_{r\to 0} \frac1{\log r}\log \left[ \sup_{ d(p,q) \le r, \, q \in C } G^+(q)  \right]= \vartheta_+,
$$
where $\vartheta_+$ is the unique positive real number satisfying
$(\lambda^u)^{\vartheta_+}   =d $.
\end{prop}

Replacing $f$ by $f^{-1}$, we see that  a similar result holds for 
$G^-$ at $\mu_C^-$-a.e. point, with $\lambda^u$ replaced by $(\lambda^s)^{-1}$
and $\vartheta_+$ by $\vartheta_-$ such that  $(\lambda^s)^{-\vartheta_-}   =d $. 
Observe that $\lambda^{u/s} \geq d$ implies that $\vartheta_\pm \in (0,1]$.
 
\medskip

If $\vartheta_+ = \vartheta_- =1$ then $|\jac(f)| = \lambda^u  \lambda^s = d \cdot d^{-1} =1$ and we are done.
 Thus we may assume that $\vartheta_- <1$. Recall that $G^+ = \alpha G^- + H$ with $\alpha >0$ and $H$ a harmonic function on $C$. 
For a $\mu^+_C$-generic point $p$,  the preceding proposition yields
\begin{eqnarray*}
\vartheta_+
&=& \liminf_{r\to 0} \frac1{\log r}\log \left[ \sup_{ d(p,q) \le r, \, q \in C } G^+(q)  \right]
\\
&=& 
\liminf_{r\to 0} \frac1{\log r}\log \left[ \sup_{ d(p,q) \le r, \, q \in C } (\alpha G^- + H) (q)  \right]~.
\end{eqnarray*}
Now observe that for any sequence $r_n \to 0$  and any sequence of points $p_n$ at distance $r_n$ from $p$
we have
$$
\liminf_n \frac{\log G^+(p_n)}{\log r_n} \ge  \liminf_n \frac1{\log r_n}\log \left[ \sup_{ d(p,q) \le r_n, \, q \in C } G^+(q)  \right] \ge \vartheta_+~.
$$
Now pick $p_n$ with $d(p_n,p) =r_n \to 0$ such that  $\frac{\log G^-(p_n)}{\log r_n} \to \vartheta_- <1$. This is possible because since $\mu_C^+$ is proportional to $\mu_C^-$ the corresponding notions of generic points coincide. 
Since $H$ is smooth, 
if      $\varepsilon$  is so small  that $(1+ \varepsilon) \vartheta_- <1$,  we infer that    
$$\alpha G^-(p_n) + H(p_n)
\ge \alpha r_n^{ (1+ \e) \vartheta_-} + O(r_n)  \ge \cst ~r_n^{ (1+ \e) \vartheta_-} $$ for large   $n$, so that 
$1 > (1+\e) \vartheta_- \ge \vartheta_+$. Since $\e$ was arbitrary we conclude that 
$\vartheta_-\geq \vartheta_+$. 

Applying the same argument with    the roles of 
 $\vartheta_+$ and $\vartheta_-$ reversed,  we conclude that  $\vartheta_+ =\vartheta_-$, which implies that 
$\lambda^u   = (\lambda^{s})^{-1}$ and finally that  $|\jac(f)| = 1$, which was the result to be proved.
\qed

\begin{rem}
It would perhaps be 
 dynamically more significant to prove that the  Hausdorff dimension of the measure $\Delta (G^+|_C)$ at a generic point is equal to the  
 Hausdorff dimension of the   measures induced by $T^+$ along
the unstable lamination.
  The  value of this dimension 
  is precisely equal to the above constant $\vartheta_+$, in virtue of  
Lai-Sang Young's formula~\cite{young}.
\end{rem}

\subsection{Proof of Proposition~\ref{prop:pesin}}\label{subs:pesin}
The proposition  relies  on the interplay between 
Pesin's theory  
and
  the laminarity properties of the currents $T^\pm$.   
This is very close in spirit to  the main results of~\cite{BLS93}.

\medskip

Since   its Lyapunov exponents  
are both  non-zero, the measure $\mu_f$ is  
hyperbolic.  
  Pesin's theory then asserts the existence of a family of Lyapunov charts, in which $f$ expands (resp. contracts)
 in the  horizontal (resp. vertical) direction. The precise statement is as follows (see~\cite[Theorem~8.14]{barreira-pesin}). 
Let $\bb(r)= \{ (x,y), \, \max \{ |x|, |y|\} < r\}$ be the polydisk of radius $r$ in $\C^2$. Then 
for any given $\varepsilon>0$, there exists an $f$-invariant set $E$ (the set of {\em regular points}) of full $\mu_f$-measure,  
a measurable  
function $\rho:  E \to (0,1)$
and  a family of charts
$\varphi_p:  \bb(\rho(p)) \to \C^2$ defined for   $p\in E$ 
and satisfying 
\begin{enumerate}[(i)]
\item
$\varphi_p (0) = p$; $e^{-\varepsilon} < \rho(f(p))/\rho(p) < e^\varepsilon$; 
\item
if  $f_p := \varphi_{f(p)}^{-1} \circ f \circ \varphi_p$, then
\begin{equation}\label{eq:pesin}
f_p(x,y) = ( a^u(p) x +  x h_1(x,y),  a^s(p) y + y  h_2(x,y))
\end{equation}
where 
$
{\lambda^u- \varepsilon} \le \abs{ a^u(p)}  \le  {\lambda^u + \varepsilon}$, 
$\lambda^s - \varepsilon  \le \abs{ a^s(p)} \le \la^s + \varepsilon$, and 
$$ \sup \{ \|h_1\|, \|h_2\|\} < \varepsilon ~;$$
\item
there exists a constant $B>0$ and a measurable function $A: E \to (0, \infty)$ such that 
\begin{equation}\label{eq:pesin2}
B^{-1} \, \| \varphi_p(q) - \varphi_p(q')\|
\le 
 \| q-q'\|
\le 
A(p)\, 
\| \varphi_p(q) - \varphi_p(q')\|
\end{equation}
with 
$e^{-\varepsilon} < A(f(p))/A(p) < e^\varepsilon$. 
\end{enumerate}
We denote  by $W^{u}_{\loc}(p)$ (resp. $W^{s}_{\loc}(p)$) the image by $\varphi_p$
of $\set{x=0}$ (resp. $\set{y=0}$). These will be referred to as local unstable (resp. stable) manifold
at $p$. 
Notice that \eqref{eq:pesin} is slightly different from the corresponding statement in 
\cite{barreira-pesin} as we have straightened the local stable and unstable manifolds. Observe also that   removing a set of measure $0$ we may assume that
$f^{-1}(W^u_{\loc}(p)) \subset W^u_{\loc}(f^{-1}(p))$,  and
 $f(W^s_{\loc}(p)) \subset W^s_{\loc}(f(p))$.

For every point $p\in E$, we define the global stable and unstable manifolds by  
$W^s(p) = \bigcup_{n\ge 0}f^{-n} W^s_{\loc}( f^n(p))$
and 
$W^u(p) = \bigcup_{n\ge 0}f^{n} W^u_{\loc}( f^{-n}(p))$.
These are embedded images of $\C$ respectively lying in $K^+$ and $K^-$
like the stable and unstable manifolds of saddle points as we saw in \S~\ref{ss:saddle}.

The next result follows from~\cite[Lemma~8.6]{BLS93}, and will be proved afterwards. Recall that  $\mu^+_C := T^+ \wedge [C]$.

\begin{lem}\label{lem:trans}
Let $E$ denote as above   the set of Pesin regular points for $\mu_f$. Then for every subset 
$A\subset E$ of full $\mu_f$-measure there exists  
$\bar  A \subset C$ of full   $\mu^+_C$-measure such that if   $\bar{p}\in \bar A$, there exists 
$p\in A$ such that 
\begin{itemize}
\item $\bar p\in W^s(p)$,
\item $W^s(p)$ intersects $C$ transversally at $\bar{p}$. 
\end{itemize}
\end{lem}

With notation  as in  the lemma, pick any $\bar{p}\in \bar{E}$ and 
introduce  the function
$$\theta_{\bar{p}} (r) = \sup\{ G^+(q), \, q\in C,\, d (\bar{p},q) \le r \}~.$$ 
To prove the proposition we need to show that $\mu_C^+$-a.s. 
$$\liminf_{r\cv0}  \frac{\log(\theta_{\bar p} (r))}{\log r}  = \vartheta_+ = \frac{\log d}{\log \la^u}.$$

Using Lemma \ref{lem:trans}, let $p\in E$ such that $W^s(p)$ intersects $C$ transversely at $\bar p$. Then
 there exists  an integer $N$ such that $f^N(\bar{p})$ lies
in $W^s_{\loc}(f^N(p))$. 
By the invariance relation for $G^+$ and the differentiability of $f$, replacing $C$ by $f^N(C)$ if needed,  it is no loss of generality to assume that $N=0$. 

Choose an integer $n$, and pick a point 
$q\in B(\rho(p))$ such that   $f^k(q) \in \varphi_{f^k(p)} B(\rho(f^k(p)))$ for all $0\le k \le n$.
Write $\varphi_p(q) = (x,y)$ so that $|x| , |y| \le \rho(p)$, and let 
$(x_k,y_k):=\varphi_{f^k(p)}^{-1} (f^k(q))$.
It then follows from \eqref{eq:pesin} that $\abs{y_n} \leq (\la^s+2\e)^n \abs{y_0}$ and 
\begin{equation}\label{eq:x}
(\la^u-2\e)^n \abs{x_0}\leq \abs{x_n} \leq (\la^u+2\e)^n \abs{x_0}.
\end{equation}Conversely, it follows a posteriori from these estimates that  any point $q = \varphi_p(x,y)$
such that 
\begin{equation*}
|x| \le  \rho(p)  ( \la^u+2\e)^{-n} e^{-\e n}
\end{equation*}
satisfies 
$|x_k| \le |x| ( \la^u+2\e)^{k} \le \rho(f^k(p))$, hence
$f^k(x) \in  \varphi_{f^k(p)} B(\rho(f^k(p)))$ for every $0\le k \le n$.

\medskip

We now estimate  $\theta_{\bar{p}} (r)$ for a given  small enough $r$. To   estimate it    from below, 
choose $q\in C$ 
such that $d(p,q)\leq r$ and 
$\theta_{\bar{p}} (r) = G^+(q)$. By \eqref{eq:pesin2}   writing  $\varphi_p(q) = (x,y)$, we infer that $\abs{x}\leq Br$. To ease notation 
let us   assume without loss of generality that $B=1$. 
 Choose $n$  such that 
\begin{equation}\label{eq:1}
\rho(p)  ( e^\e(\la^u+2\e))^{-n}   \leq 
r 
<\rho(p)  ( e^\e(\la^u+2\e))^{-(n-1)}.
\end{equation}
From \eqref{eq:1} and the invariance relation for $G^+$  we get the upper bound:
\begin{align*}
\log \theta_{\bar{p}} (r) &=  
- \log (d^n) + \log G^+(f^n(q))
  \le 
- \log (d^n) +   \max_{\varphi_p(B(\rho(p)))}\nolimits G^+  \\ 
&\le - n\log d + \cst
\le 
   \frac{\log r - \log \rho(p)}{ (\e +\log (\la^u+ 2 \varepsilon))} \log d+ \cst~,
\end{align*}
where the first inequality on the second line follows from the fact that due to \eqref{eq:pesin2},
$\bigcup_E \varphi_p(B(\rho(p))$ is bounded in $\C^2$.
Letting $r\to 0$, 
we infer that 
$$\liminf_{r\to 0} \frac{\log \theta_{\bar{p}}(r)}{\log r} \ge 
\frac{\log d}{\e+\log (\lambda^u + 2 \varepsilon)}~.$$
Since this holds for every $\varepsilon$, we conclude  that 
\begin{equation}\label{eq:lower}
\liminf_{r\to 0} \frac{\log \theta_{\bar{p}}(r)}{\log r} \ge 
\frac{\log d}{\log \la^u }~.
\end{equation}

To prove the opposite inequality we proceed as follows. 
Let us introduce the auxiliary function
$\psi: p\mapsto  \sup_{W^u_{\loc}(p)} G^+$. This is a measurable function that is uniformly bounded
from above since
$\bigcup_E \varphi_p(B(\rho(p))$ is bounded in $\C^2$.
Likewise $\psi(p)>0$ for every $p\in E$ since $W^u_{\loc}(p)$ cannot be contained in 
$K^+$, see~\cite[Lemma 2.8]{BLS93} (the argument is identical to that of Proposition~\ref{prop:stable and G-}).

Now fix a constant $g_0>0$ such that the set 
$\{\psi > g_0\}$ has positive $\mu_f$-measure. 
By the Poincar\'e recurrence theorem,
there exists a measurable set 
 $A\subset E$ of full $\mu_f$-measure 
such that for every $p\in A$, $\psi (f^{n_j}(p)) >g_0$ for 
infinitely many  $n_j$'s.

Let $\bar A$ be as in Lemma \ref{lem:trans}, let $\bar p \in \bar A $ and $p$ be as above.
Stable manifold theory shows that there exists $n_0$ such that for $n\geq n_0$
 the connected component 
 of $\varphi_{f^n(p)}^{-1}(f^n(C))$ in $\bb(\rho(f^n(p)))$
  containing $ \varphi_{f^n(p)}^{-1} (\bar p)$ is a graph over the first coordinate which converges exponentially fast in the $C^0$ (hence $C^1$) topology  to $\set{y = 0}$ 
  \cite[\S 8.2 and Thm 8.13]{barreira-pesin}. Let us denote it by $C_{n, f^n(\bar p)}$.

Since $p$ belongs to $A$ we have $\psi (f^{n_j}(p)) >g_0$ for 
infinitely many  $n_j$'s.
To ease notation we simply write  $n$ for $n_j$. 
  Since $G^+$ is H\"older continuous,  for such an iterate $n$ we infer that 
 \begin{equation*}
 \sup\{ G^+(w), w\in {C_{n, f^{n}(\bar p)}}\}  \geq g_0 - \delta_n,
 \end{equation*} where $\delta_{n}$ is exponentially small. Let $w_{n}\in 
    {C_{n, f^{n}(\bar p)}}$  be a point at which  $G^+ (w_n)\geq \frac{g_0}{2}$.   Consider now $f^{-n}(w_{n})$ 
  and denote $\varphi_p^{-1} (f^{-n}(w_{n})) = (x_{n}, y_{n})$. By \eqref{eq:pesin}, we have that 
  $$ 
\abs{x_{n}}\leq  \frac{\rho(f^{n}(p))}{(\la^u-2\e)^n}\le \frac{\cst}{(\la^u-2\e)^n},$$ hence, since $C$ is transverse to $W^s_{\rm loc}(p)$, 
from \eqref{eq:pesin2} we get that 
$$d (f^{-n}(w_{n}), \bar p) \leq  \frac{C_0}{(\la^u-2\e)^n}, $$ where $C_0$ 
does not depend on $n$. 
Therefore, putting $r_n =  \frac{C_0}{(\la^u-2\e)^n}$ and using  the invariance relation for $G^+$ and the 
definition of $w_n$, 
 we infer that 
$$ \theta(\bar p, r_n) \geq \frac{g_0}{2d^n} = \frac{g_0}{2} 
 \left(\frac{r_n}{C_0}\right)^{\frac{\log d}{\log(\la^u-2\e)}}.$$ Finally 
 $$\limsup_{n\cv\infty} \frac{\log\theta(\bar p, r_n)}{\log r_n} \leq  \frac{\log d}{\log(\la^u-2\e)},$$ thus 
 $$\liminf_{r\cv0}  \frac{\log\theta(\bar p, r )}{\log r } \leq  \frac{\log d}{\log \la^u }$$ which, along with \eqref{eq:lower}, finishes the proof.

\begin{proof}[Proof of Lemma~\ref{lem:trans}]
The proof relies on the theory of laminar currents \cite{BLS93}.
It is shown in~\cite[Thm 7.4]{BLS93} that 
 the positive closed 
 $(1,1)$-current $T^+:=  dd^c  G^+$ is   laminar. 
 
Recall that this means the following. First, we say that 
a current $S$ in $\om\subset \cd$  is {\em locally uniformly laminar}  if every point in $\supp(S)$ admits a neighborhood $B$
biholomorphic to a bidisk, such that in adapted coordinates, $S$ locally writes as   $\int  [\Delta_a]\, d\alpha(a)$, where the 
$\Delta_a$ are disjoint graphs over the first coordinate in $B$ and $\alpha$ is a positive measure on the space of such graphs.
  These disks will be said to be {\em subordinate} to $S$. Notice that a locally uniformly laminar current is always closed.

A current is   {\em laminar} if 
for any   $\varepsilon >0$ there exists a finite  family of disjoint  open sets 
$\om ^i$, and  for each $i$ a locally uniformly laminar current 
$T^i\leq T$  
such that  the mass of $T -  \sum_i T^i$ is    smaller than $\varepsilon$.
If $R$ is any positive closed current in $\C^2$ such that the wedge product $  T\wedge R$ 
is admissible, then slightly abusing notation we define the wedge product  $\big(\sum_i T^i\big) \wedge R$ by  $\sum_i (T^i\wedge R)\rest{\om^i}$. 

The geometric intersection product of a current of integration over a curve $[M]$ with a uniformly laminar current 
$T = \int [\Delta_a] d\alpha(a)$ is defined by 
$$
T \geom [M] = \int   [\Delta_a \cap M] d\alpha(a)
$$
where $[\Delta_a \cap M]$ is the atomic positive measure 
putting mass $1$ at any  intersection of 
$\Delta_a$ and  $M$. 
If $T$ has continuous potential then $T\geom [M] = T\wedge [M]$, see \cite[Lemma 6.4]{BLS93}. 
 
\medskip

Pick any open subset  $\om\subset \cd$, and let
$0<S^+\leq T^+$ be any locally uniformly laminar current in $\om$.
Denote by $\mathbf{M}(\mu)$  the total mass of a given positive measure $\mu$.
We claim that
\begin{equation}\label{eq:mass}
\lim_{n\to \infty} \mathbf M \lrpar{ \frac{1}{d^n} (f^n)^* S^+ \geom [C]}=\deg(C)\cdot  \mathbf M (S^+\wedge T^-) ~.
\end{equation}
To see this, we observe that  the current $S^+$ has continuous potential by \cite[Lemma 8.2]{BLS93} so that 
$$\frac1{d^n} (f^n)^* S^+ \geom [C]= \frac1{d^n} (f^n)^* S^+ \wedge [C]$$
as positive measures in $f^{-n}(\om)$. Now $f^n : f^{-n}(\om)\to \om$ is an automorphism so that 
$$
\mathbf M \lrpar{ \frac{1}{d^n} (f^n)^* S^+ \geom [C]} =  \mathbf M  \lrpar{ S^+ \wedge d^{-n}(f^n)_*[C]}~.
$$
Replacing $C$ by some iterate we may assume that it intersects the line at infinity at $p_+$ only, hence $d^{-n}(f^n)_*[C] \cv (\deg(C))T^-$ by \cite{BS1,fornaess-sibony}.
Again since  $S^+$ has continuous potential in $\om$, the measures 
$S^+\wedge d^{-n}(f^n)_*[C] $ converge to  $\deg(C) \, S^+\wedge T^-$. We conclude that 
$\mathbf M \lrpar{ {d^{-n}} (f^n)^* S^+ \geom [C]}$ converges to   $\deg(C)\,  \mathbf M (S^+\wedge T^-)$ as $n\cv\infty$ as required. 

\smallskip

Another result in \cite{BLS93} is that   
 $T^+$ and $T^-$ intersect  geometrically (see also \cite{isect}). 
 This implies  that  for every $\e>0$ there exists a 
current $T^+_\e\leq T^+$, which is a finite sum of uniformly laminar  currents in disjoint open sets $\om^i$ as  above,   
and such that 
 $\mathbf{M}(T^+_\e\wedge T^-) \geq 1-\e$. Then from~\eqref{eq:mass}  we deduce that  for large $n$, the positive measures
 $$\frac1{d^n} (f^n)^* T^+_\e \wedge [C] 
:= \sum_i \frac1{d^n} (f^n)^* T^+_\e \wedge [C] |_{f^{-n}\om_i}
$$ are dominated by $\mu_C^+$, and
\begin{multline*}
\mathbf{M} \left(\frac1{d^n} (f^n)^* T^+_\e \wedge [C]\right) \ge 
(1-\e) \deg(C)\, \sum_i \mathbf{M} (T_\e|_{\om_i} \wedge T^{-}) \ge\\ \ge (1-\e)^2\deg(C)\mathbf M(T^+\wedge T^{-})  \ge (1-2\e)\, \deg(C)~.
\end{multline*}
Now recall that $T^+_\e|_{\om_i}$ has continuous potential for each $i$ hence does not
charge any curve. By \cite[Lemma 6.4]{BLS93} it follows that 
only transverse intersections between disks subordinate to $\frac1{d^n} (f^n)^* (T^+_\e|_{\om_i})$ and $[C]$ need to be taken into account in the computation of the geometric intersection $\frac1{d^n} (f^n)^* (T^+_\e|_{\om_i}) \geom [C]$.  
Further by~\cite[Corollary~8.8]{BLS93},  almost every disk subordinate to $T^+_\e|_{\om_i}$
is an open set of some stable curve $W^s(p)$ 
for some $p\in A$.

In particular, there exists a set $B_n$ of total mass for the positive measure $d^{-n} (f^n)^* T^+_\e \wedge [C]$ such that for all points $q \in B_n \subset C$ 
there exists a point $p\in A$ such that $W^s(p)$ intersects $C$ transversally at $q$. Since
$\mu^+_C( B_n) \ge (1-2\e)\, \deg(C)$
the proof is complete.
\end{proof}

%

\subsection{A uniform Theorem~\ref{thm:maineffective}}\label{subs:uniformA'}

In this section we indicate how our  arguments can be modified so as to get the following statement.

 \begin{theoseconde}\label{thm:maineffective-plus}
Let  $f$ be  a polynomial automorphism of H\'enon type  of the   affine plane, defined over a number field $\mathbb{L}$. 
Assume   that there exists an archimedean place $v$ such that $|\jac(f)|_v \neq 1$.

For any integer $d$, there exists a positive constant $\e(d) >0$ and an integer $N(d) \ge1$ such that  for any algebraic curve $C$ defined over $\mathbb L$  of degree at most $d$,
 the set $\{ p\in C(\LL^{\rm alg}), \, h_f(p) \le \e(d)\}$ contains at most $N(d)$ points.
\end{theoseconde}

\begin{proof}
Pick an automorphism $f$ of H\'enon type defined over $\LL$ and 
fix any archimedean place $v$.
We suppose that there exists a sequence of curves $C_m$ defined over $\LL$ of degree $d$ and finite sets $F_m \subset C_m$ invariant under the absolute Galois group of $\LL$ such that $\#F_m\geq m$ and $h_f(F_m) \le \frac1m$. Let us show that $|\jac(f)|_v =1$.

Let $\mu_m$ be the  probability measure  equidistributed over $F_m$.

\begin{lem}\label{lem:easy}
Any weak limit of the sequence $(\mu_m)$ is supported on a curve of degree $d$.
\end{lem}

Suppose first that   any curve defined over $\LL$ intersect only finitely many $F_m$'s. Then Yuan's result~\cite[Theorem 3.1]{yuan} implies   the equidistribution $\mu_n \to \mu_{f,v}$ (see \cite[Thm B]{lee}). However $\mu_{f,v}$  gives no mass to curves so  
 the previous lemma gives a contradiction. 

We may thus suppose that there exists a curve $D$ defined over $\LL$ that contains infinitely many $F_m$'s. 
Theorem~\ref{thm:mainArch} then applies to show  that $|\jac(f)|_v =1$ as required.
\end{proof}

\begin{proof}[Proof of Lemma~\ref{lem:easy}]
For each $m$ pick an equation $P_m = \sum_{i+ j \le d} a_{ij}^m x^i y^j$ of $C_m$ such that 
$\max \{ | a^m_{ij} | \} = 1$.
Replacing $F_m$ by a suitable subsequence we may assume that each coefficient $a_{ij}^m$ converges to some $a_{ij} \in \LL_v$ and we set $P = \sum_{i+ j \le d} a_{ij} x^i y^j$.
Since the height of $F_m$ is bounded from above, $\bigcup_m F_m$ is included in a fixed bounded set $K$ in $\LL^2_v$. We have $\sup_K |P_m - P| \to 0$ and this implies
$\int |P| d\mu = \lim_m \int |P_m| d\mu_m = 0$. Therefore $\mu$ is supported on the curve $\{P= 0\}$.
\end{proof}

\section{The DMM statement under a transversality assumption}\label{sec:transversality}

This section is devoted to the proof of Theorem \ref{thm:mainNA} in the number field case.  
Let $f$ be a regular polynomial automorphism of $\mathbb{A}^2$,
and $C$ be an irreducible algebraic curve containing 
infinitely many periodic points, both defined over a number field $\LL$. 
By the transversality assumption (T), 
replacing $f$ by $f^N$ if needed we assume that one of these periodic points $p\in \mathrm{Reg}(C)$ is fixed 
and satisfies $Df_p(T_pC) \neq T_pC$.

 We want to show that $\jac(f)$ is a root of unity. 
To do so it will be enough to prove that $|\jac(f)|_v = 1$ for each place $v$.

If the place $v$ is  archimedean, the equality $|\jac(f)|_v = 1$ follows from Theorem~\ref{thm:mainArch}, so we will work at non-Archimedean places only.

\begin{lem}\label{lem:thisone}
Let $f$ and $C$ be as in Theorem \ref{thm:mainNA}. 
Let $p$ be any fixed point lying on $C$ and denote by $\lambda_1, \lambda_2$ the two (possibly equal) eigenvalues of $Df(p)$.

At any non-Archimedean place $v$, 
then either $|\lambda_1|_v = |\lambda_2|_v = 1$ or $p$ is a saddle.
\end{lem}

It is also not difficult to see that for all but finitely many places $v$, 
\emph{all} periodic points are indifferent in the sense that their multipliers have norm $1$.

\begin{proof}[Proof of Lemma~\ref{lem:thisone}]
Assume by contradiction that $\abs{\lambda_1}_v <1$ and $|\lambda_2|_v \leq  1$.

Since $\abs{\lambda_1}_v \le |\lambda_2|_v \le 1$ it is classical that 
there exists a (bounded) neighborhood $U$ of $p$ that is forward invariant, 
 in particular $U\subset K^+$. Indeed, performing a linear change of coordinates we can write 
 $f(x,y)  =  (\la_1 x + \hot, \la_2 y + \hot)$, and since $v$ is non-Archimedean if $x$ and $y$ are small enough, then 
 $\abs{\la_1 x + \hot} = \abs{\la_1}\abs {x}$ and   $\abs{\la_2 y + \hot} =  \abs{\la_2}\abs {y}$.

It follows that $G^+\rest{U} \equiv 0$, hence from Proposition~\ref{prop:main2} we deduce that 
 $G^-\rest{U\cap C}$ is harmonic, and since $G^-(p) = 0$ and $G^-\geq 0$, by Proposition \ref{prop:maximum-principle} 
 we conclude that 
 $G^-\rest{U\cap C} \equiv 0$ as well. 
This implies that $f^{-n}  (U\cap C) \subset K$ for all $n$, and since $K$ is bounded the 
Cauchy inequality implies that the norms  $\| Df^{-n}(p)\|$ stay  uniformly bounded in $n$ along $U\cap C$. 

If $p$ is a sink, that is $|\lambda_1|_v \leq |\lambda_2|_v <  1$, then $\| Df^{-n}(p)\|$ must grow exponentially and
we readily get a contradiction.
 The semi-attracting case $\abs{\lambda_1}_v< \abs{\lambda_2}_v =1$ 
requires a few more arguments. 

Assume first that $\lambda_2$ is not a root of unity. Then a theorem by 
Herman and Yoccoz \cite{HY83} asserts that  in this case 
$\lambda_2$ satisfies a Diophantine condition, hence so does the pair $(\lambda_1, 
\lambda_2)$, and the fixed point $p$ is analytically linearizable. Therefore there exist adapted coordinates $(x,y)$ near $p$, 
in which $p$  is sent to the origin, and  $f$ takes the form $f(x,y) = (\lambda_1 x, \lambda_2 y)$. Since $\| Df^{-n}(p)\|$ is  uniformly bounded in $n$ 
along $U\cap C$, in these coordinates $C$ must be tangent to the $y$ axis, so it locally expresses as a graph 
$t\mapsto (\psi(t), t)$.  Now $Df^{-n}(\psi'(t), 1) = (\lambda_1^{-n}\psi'(t), \lambda_2^{-n})$ is unbounded as 
$n\cv \infty$ unless $\psi\equiv 0$. From this we 
conclude that $C$ locally coincides with the $y$ axis, so in particular it is invariant. But $f$ does not admit any invariant algebraic curve so we reach 
a contradiction. 

If $\lambda_2$ is a root of unity, the argument is similar. Replace $f$ by some iterate so that $\lambda_2=1$. 
To ease notation we work   with $f^{-1}$ instead of $f$.
A theorem by Jenkins and Spallone 
\cite[\S 4]{jenkins spallone} asserts that there are coordinates $(x,y)$ as above in which $f^{-1}$ expresses as
$$f^{-1}(x,y) = (\lambda_1^{-1} x (1+ g(y)), h(y)), \text{ with } g(0) = 0 \text{ and } h(y) = y + \hot$$ (recall that $\abs{\la_1^{-1}}>1$). 
From this we deduce that 
$$ f^{-n}(x,y) = (\lambda_1^{-n}x (1+ g_n(y)), h^{n}(y)) = \lrpar{\lambda_1^{-n}x \prod_{j=0}^{n-1} (1+ g(h^{j}(y))) , h^{n}(y)}.$$
For $\e$ small enough, if $\abs{y}_v <\e$,  using the ultrametric inequality we   get 
  that  for every $0 \leq j\leq n$, $\abs{h^j(y)}_v <\e$. 
It follows that $\abs{1+ g(h^j(y))}_v =1$.

Now as in the previous case, in the new coordinates 
 the curve $C$ must be tangent to the $y$-axis. Parameterize it as $t\mapsto (\psi(t), t)$, so that the curve $f^{-n} C$ is parameterized by 
 $$ t \longmapsto   \lrpar{\lambda_1^{-n}\psi(t) \prod_{j=0}^{n-1} (1+ g(h^j(t))) , h^n(t)}.$$ We see that the only possibility
  for it to be locally bounded 
 as $n\cv \infty$ is that $\psi\equiv 0$. As before we deduce that $C$ is equal to the $y$-axis, which is invariant, and again we get a contradiction. 
\end{proof}

Let us resume the proof of Theorem~\ref{thm:mainNA}. Pick any non-Archimedean place $v$, 
and recall that we wish to prove that
$\abs{\jac(f)}_v=1$.
To simplify notation we   drop all indices referring to this place $v$. 
 Let $p$ be the fixed point satisfying (T), and denote by $\lambda_i$  its eigenvalues. Then $\jac(f)  = \lambda_1\lambda_2$.
  By the previous proposition 
either $\abs{\lambda_1} = \abs{\lambda_2} =1$ and we are done, or $p$ is a saddle.
For   notational consistency we   
  denote  by  $u$ (resp. $s$)  the unstable (resp. stable) eigenvalue (which can alternatively be $\la_1$ or $\la_2$ depending on the place).  By the transversality assumption (T), $W^u_{\rm loc}(p) $ and $W^s_{\rm loc} (p)$ are not tangent to $C$ at $p$. Indeed since 
  the tangent directions to $W^u_{\rm loc}(p) $ and $W^s_{\rm loc} (p)$ are given by the eigenvectors of $Df(p)$, this transversality 
  does  not depend on the place. 
  
By Lemma \ref{lem:coordinates} there are adapted coordinates $(x,y)$ near $p$
 in which $f$ takes the form 
\begin{equation}\label{eq30}
f(x,y) =\left( ux(1+xy g_1(x,y)), sy(1+xy g_2(x,y)) \right).
\end{equation}
By scaling  the coordinates  if necessary we may assume that   we work in the unit bidisk $\bb$.

The following key renormalization lemma  will be proven afterwards.

 \begin{lem}\label{lem:renorm}
If $\abs{x}, \abs{y}$ are small enough, then for every $1\leq j\leq n$, 
$f^j\lrpar{\frac{x}{u^n}, {y}}  \in  \bb$ and 
$$f^n\lrpar{\frac{x}{u^n},  {y}}  = (x, 0) + O(n\rho^n),$$ 
uniformly in $(x,y)$,  with $\rho := \max \{ |u|^{-1} , |s| \} <1$.
\end{lem}

In the coordinates $(x,y)$, we write $C$ as a graph $y=\psi(x)  = bx + \hot$ over the first coordinate. 
Replacing $y$ by $by$ we can assume $b=1$.
Using Proposition \ref{prop:main2} we set    $\tilde{G}(x) = G^+(x, \psi(x)) = \alpha G^-(x, \psi(x))  + H(x, \psi(x))$.  
 By Lemma \ref{lem:renorm}, the continuity of $G^+$ and the invariance relation for $G^+$  we get that for small enough $x$, 
 \begin{equation}\label{eq:G+}
 d^n \tilde{G}\lrpar{\frac{x}{u^n}} = G^+\circ f^n\lrpar{\frac{x}{u^n}, \psi\lrpar{\frac{x}{u^n}}}\cv G^+(x,0).
 \end{equation}
Applying Lemma \ref{lem:renorm} to $f^{-1}$, for small $y$ we have 
the  following uniform convergence in a small disk:
\begin{equation}\label{eq:34}
f^{-n} (\psi^{-1}(s^{n} y), s^{n} y) \underset{n\cv\infty}\longrightarrow (0,y) ~.
 \end{equation}
Now we  claim that $|us| =1$. 
Indeed   assume  by contradiction that $|us|>1$. Then putting
$y_n =  s^{-n} \psi(x/u^n)$, we get that $y_n\cv 0$ when $n\cv\infty$. 
 We write 
\begin{align}\label{eq:35}
d^n\tilde{G}\lrpar{\frac{x}{u^n}}  = \alpha G^-\! \circ f^{-n}\lrpar{\frac{x}{u^n}, \psi\lrpar{\frac{x}{u^n}}} + d^{n} H 
\lrpar{\frac{x}{u^n}, \psi\lrpar{\frac{x}{u^n}}}  
 \end{align}
and 
applying \eqref{eq:34}, we see that 
$$G^-\! \circ f^{-n}\lrpar{\frac{x}{u^n}, \psi\lrpar{\frac{x}{u^n}}}  = G^-\! \circ f^{-n} (\psi^{-1}(s^{n} y_n), s^{n} y_n) \underset{n\cv\infty}\longrightarrow G^-(0,0) = 0.$$ Thus from  \eqref{eq:G+} 
and \eqref{eq:35} we infer that 
$$d^{n} H 
\lrpar{\frac{x}{u^n}, \psi\lrpar{\frac{x}{u^n}}} \underset{n\cv\infty}\longrightarrow G^+(x,0) ~,$$ locally uniformly in the neighborhood of the origin.
Since a limit of harmonic functions is harmonic (see Proposition \ref{prop:uniform-harmonic}), we conclude that 
$G^+$ is harmonic, hence identically zero on $W^u_{\rm loc} (p)$, 
 thereby 
contradicting  Proposition~\ref{prop:stable and G-}. This contradiction shows that 
 at the place $v$ we have $\abs{us}_v = \abs{\jac(f)}_v =1$,
 and completes the proof of Theorem \ref{thm:mainNA}.
\qed

 \begin{proof}[Proof of Lemma~\ref{lem:renorm}]
Recall that   we work in the unit bidisk $\bb$. Scaling the coordinates further, 
we may assume that the functions $g_1$, $g_2$ appearing in \eqref{eq30} are   as small as we wish,  say 
$\sup_\bb \{|g_1|, |g_2| \}\leq \e$, where $\e$ is a small positive constant whose value will be determined shortly. 

 Assume $(x_0, y_0)\in \bb$  and denote by 
$(x_1,y_1) = f(x_0,y_0), \ldots, (x_k,y_k) = f^k(x_0,y_0)$ its sucessive iterates
(whenever defined). 
Using~\eqref{eq30} recursively, 
we obtain
$$x_k = u^k x_0  \prod_{j=0}^{k-1} \big(1 + x_jy_j g_1(x_j,y_j)\big)\text{ and }\\
y_k = s^k y_0  \prod_{j=0}^{k-1} \big(1 + x_jy_j g_2(x_j,y_j)\big).
$$
We claim that if $|x_0| \le B|u|^{-n}$  for a suitable constant $B$ and $|y_0 | \le 1$ then 
the $n$ first iterates of $(x_0,y_0)$ are well-defined.

Indeed assume by induction that  the $k-1$ first iterates of $(x_0,y_0)$ stay in $\bb$ for some 
 $k\le n-1$. Then   we get that 
$$ |y_k| \le |s|^k \prod_{j=0}^{k-1} (1 + |y_j|  \varepsilon) 
.$$
This will in turn be bounded by $A \abs{s}^k$ if  $A$ is any constant satisfying $A \ge\prod_{j\ge 0} (1 + A |s|^j  \varepsilon)$. 
We leave the reader check that if $\e <(1-\abs{s})/10$, then $A=3$ will do.  In what follows we work under this assumption. 

Now assume that  $|x_0| \le \frac14 |u|^{-n}$ and let us show by induction that for small enough $\e$, 
$\abs{x_j}\leq \abs{u}^{j-n}$ for $0\leq j\leq n$  (so that in particular $(x_j, y_j)\in \bb$). 
Indeed, if this estimate holds for $0\leq j\leq k-1$, then using the formula for   $x_k$, we get that 
\begin{align*}
|x_k| &\le |u|^k |x_0| \prod_{j=0}^{k-1} (1 +    |x_j| 3  |s|^j   \varepsilon) 
\leq \frac14 \abs{u}^{k-n}  \prod_{j=0}^{k-1} \lrpar{1 + 3 \e \frac{\abs{us}^j}{\abs{u}^n}} \\
&\leq \frac14 \abs{u}^{k-n}\lrpar{1+ \exp\lrpar{\frac{3\e}{\abs {u}^n} \sum_{j=0}^n \abs{us}^j} }\\
&\leq \frac14 \abs{u}^{k-n}\lrpar{1+ \exp\lrpar{ {3\e} \sup_{n\geq 0}
 \max\lrpar{\frac{\abs{s}^n}{\abs{us}-1}, \frac{n}{\abs {u}^n} }}}. 
\end{align*}
Hence, choosing $\e$ sufficiently small (depending only on $u$ and $s$), the exponential term is smaller than 2, and we are done.

To get the conclusion of the lemma, we simply reconsider the previous computation for $k=n$, and use the inequality 
$$\abs{\prod (1+z_j) -1}\leq  \exp\lrpar{\sum \abs{z_j}}-1 $$ to obtain that 
\begin{align*}
\abs{\frac{x_n}{u^n x_0}   -1} &= \abs{ \prod_{j=0}^{n-1} \big(1 + x_jy_j g_1(x_j,y_j)\big) -1}  = O\lrpar{
 \max\lrpar{ {\abs{s}^n} , \frac{n}{\abs {u}^n} }}
\end{align*}
and we are done.
 \end{proof}

In the next theorem, we give a direct argument for Theorem \ref{thm:mainNA} under a more restrictive assumption
 which is reasonable from the dynamical point of view. We feel interesting to include it as it gives in this case a purely 
 archimedean proof of our main result. Observe that no transversality assumption is required.

\begin{thm}\label{thm:saddle archimedean}
Let $f$ be a polynomial automorphism of the affine plane of H\'enon type that is defined over a number field $\LL$. 
Assume that there exists an algebraic curve defined over $\LL$ and containing infinitely many periodic points. 

Suppose that at some archimedean place there exists a saddle point 
$p\in C$. Then $\jac(f)$ is a root of unity. 
\end{thm}

It follows from~\cite{bls2} 
that at the  archimedean place most periodic orbits of  $f$ are saddles, which make the assumption of the proposition natural. 
Still, there exists examples   polynomial automorphisms of $\cd$ with infinitely many non-saddle periodic orbits, even in a conservative setting, see \cite{duarte ETDS2008}.

\begin{proof}
We do an   analysis similar to that  of the proof of Theorem \ref{thm:mainNA}, starting from equation 
\eqref{eq:G+}, and  keeping the same notation. 
For simplicity  we write $\LL_v = \C$, and drop the reference to $v$. 
By assumption there is  a saddle periodic point $p\in C$,  with multipliers $u$ and $s$. From 
 Theorem \ref{thm:mainArch} we know that $\abs{us}= 1$. 
We work in the local adapted coordinates 
$(x,y)$ given by Lemma \ref{lem:coordinates}.
Since we make no smoothness or transversality assumption here, we
pick any local irreducible component  of $C$ at $p$,
and parameterize it by 
 $\Psi: t \mapsto (t^k, \psi(t))$ with $\psi(t) = t^l + \hot$. By Proposition \ref{prop:main2} for small $t\in \C$ we have that 
 $$\tilde G(t)  := G^+\circ \Psi(t) =G^+(t^k, \psi(t)) = \alpha G^- (t^k, \psi(t)) + H(t^k, \psi(t)).$$
Swapping the stable and unstable directions if needed we may assume that $k \leq l$.
Pick a $k^{\rm th}$ root of $u$, denoted by $u^{1/k}$. 

Applying the same reasoning as in \eqref{eq:G+} we get that 
 \begin{equation}\label{eq:G+singular}
 d^n \tilde{G}\lrpar{\frac{t}{u^{n/k}}} = G^+\!\circ f^n\lrpar{ \frac{t^k}{u^n} , \psi\lrpar{\frac{t}{u^{n/k}}}}\cv G^+(t^k,0).
 \end{equation}
Since $k\leq l$ and $\abs{us}= 1$, we see  that $\abs{s^{k} u^l}\geq 1$, from which we infer that 
  $\psi\lrpar{\frac{t}{u^{n/k}} }= O(s^n)$. Therefore  we can do the same with  $f^{-n}$ to  deduce that 
\begin{align} \label{eq:G-singular} 
d^n \tilde{G}&\lrpar{\frac{t}{u^{n/k}}} = \alpha G^- \!\circ f^{-n}\lrpar{ \frac{t^k}{u^n} ,\psi\lrpar{\frac{t}{u^{n/k}}}} 
+ d^n H  \lrpar{ \frac{t^k}{u^n} ,\psi\lrpar{\frac{t}{u^{n/k}}}} 
\\ \notag
&= \alpha G^- \lrpar{0, s^{-n}  \psi\lrpar{\frac{t}{u^{n/k}} }} +o(1) + 
d^n H \lrpar{ \frac{t^k}{u^n} ,\psi\lrpar{\frac{t}{u^{n/k}}}} \\   \notag
&= \alpha G^- \lrpar{0, \frac{t^l}{(s^ku^l)^{n/k}}} +o(1) + d^n H  \lrpar{ \frac{t^k}{u^n} ,\psi\lrpar{\frac{t}{u^{n/k}}}}  . 
\end{align}
Arguing exactly as in the proof of Theorem \ref{thm:mainNA}, we see that this is contradictory unless $\abs{s^k u^l} = 1$, that is, $k=l$ 
(in particular if $C$ is smooth at $p$ it must be transverse to $W^u_{\rm loc}(p)$ and $W^s_{\rm loc} (p)$).

\medskip

Now since we work in the archimedean setting, we can push the analysis further and proceed to prove that $us = \jac(f)$ is a root of unity.  
Assume by contradiction that this is not the case. 
Choose any $\theta$ in the unit circle, and pick a subsequence $(n_j)$ such that  
$(u^ks^k)^{n_j} \to \theta$. 
Observe that $s^{-n} \psi (x/u^{n/k}) \to 
x^k/\theta$.
Then by~\eqref{eq:G+} and \eqref{eq:35} in the smooth case ($k=1$) and 
\eqref{eq:G+singular} and \eqref{eq:G-singular} in the singular case  we get that for small $t$
\begin{align*}
G^+(t^k,0) &= \lim_{n_j \to \infty}  G^+\! \circ f^{n_j}\left(\Psi \lrpar{\frac{t}{u^{n_j/k}} } \right)\\
&= \lim_{n_j \to \infty} \left[\alpha   G^-\! \circ f^{-n_j}\left( \Psi \lrpar{\frac{t}{u^{n_j/k}} } \right)
+ d^{n_j} H  \circ \Psi \lrpar{\frac{t}{u^{n_j/k}} } 
\right]\\
& =  \alpha G^- \lrpar{0, \frac{t^k}{\theta}}  + 
 \lim_{n_j \to \infty} d^{n_j} H \circ  \Psi \lrpar{\frac{t}{u^{n_j/k}} }  
\end{align*} 
Since $\theta$ was arbitrary and since a uniform limit of harmonic functions is harmonic, we see that the Laplacian of the function  
$t \mapsto G^- (0, t^k)$  is rotation-invariant in a neighborhood of the origin. 
Observe that  this Laplacian expresses as 
 $\kappa^* \Delta (G^-(0,t))$, where $\kappa: t\cv t^k$. Recall also that the support of $\Delta (G^-(0,t))$ equals 
 $\fr (K^-\cap W^s_{\rm loc}(p))$, where the boundary is relative to the intrinsic topology on $W^s_{\rm loc} (p)$. 
 Thus we  conclude  that relative the linearizing coordinate on $W^s(p)$,   
$\kappa^{-1} (\fr (K^-\cap W^s_{\rm loc}(p)))$ is rotation invariant, so $\fr (K^-\cap W^s_{\rm loc}(p))$ is rotation invariant as well. 
But   since  $dd^c(G^-\rest{W^s_{\rm loc} (p)})$ 
gives no mass to points, 
   $p$ must be an accumulation point of $\fr (K^-\cap W^s_{\rm loc}(p))$. By rotational invariance,  $K^-\cap W^s_{\rm loc}(p)$ will then contain small circles around the origin. By the Maximum principle this
    implies that $G^-\rest{W^s_{\rm loc} (p)}$ vanishes in a neighborhood of $p$, which  contradicts Proposition 
   \ref{prop:stable and G-}. The proof is complete. 
 \end{proof}

\begin{rem}\label{rem:symmetry}
It is a well-known idea in the dynamical study of plane polynomial automorphisms 
that the slices of $T^\pm$ by stable and unstable manifolds (or more generally by any curve) 
contain a great deal of information about $f$. 
For instance, as we saw in \S \ref{sec:dissipative}, the  Lyapunov exponents of the maximal entropy measure can be read off this data. 
The same holds  for  multipliers of all saddle periodic orbits. 
 See also  \cite{bs6} for a striking application of this circle of  ideas. 

The proof of Theorem \ref{thm:saddle archimedean}  (with $k=1$, say) implies 
 that   in adapted coordinates a  relation of the form  $G^+(x,0) = G^-(0,x) + \tilde H$ holds, where 
 $\tilde H$ is a harmonic function. 
So we get  that  an unstable  slice of $K^+$ is {\em holomorphically equivalent} to  a stable slice of $K^-$. 
 
 This rigidity    suggests a strong form of symmetry between $f$ and $f^{-1}$ which 
  gives additional credit to Conjecture \ref{conj:MM2}.
\end{rem}


\section{Conclusion of the proof of Theorems \ref{thm:mainArch} and \ref{thm:mainNA}} \label{sec:general fields}

Recall that when $f$ and $C$ are defined over a number field,  
Theorems \ref{thm:mainArch} and \ref{thm:mainNA} were established in  \S \ref{sec:dissipative} and \S \ref{sec:transversality} respectively. 
In this section we explain how a specialization argument allows to 
extend these results to an arbitrary field $K$ of  characteristic zero. 
Our approach  treats Theorems \ref{thm:mainArch} and \ref{thm:mainNA} simultaneously. 

Notice that   
it is not clear how to use the Lefschetz Principle here
because the statement that $f$ possesses infinitely many periodic points on a curve does not belong to first order logic. 

\medskip

Since $K$ has characteristic zero,  replacing it by an algebraic extension if needed,
 we may assume that it contains the algebraic closure of its prime field. We fix an isomorphism of this algebraically closed field with $\Q^{\rm alg}$. 

We first  make a conjugacy so   that $f$ becomes a regular polynomial automorphism of degree $d \ge 2$.
Pick a  finitely generated $\Q^{\rm alg}$-algebra $R \subset K$ containing all coefficients
of $f$, $f^{-1}$, and of  an equation defining $C$. We may assume that ${\rm Frac}(R) = K$, and we set $S = \spec (R)$. 

Let us start with a  loose explanation of the proof. The field $K$ is a  function  field 
over $\Q^{\rm alg}$, which we view as the function field of the variety $S$. 
For every $s\in S$, we   substitute the corresponding value $s$ into the coefficients of $f$, obtaining a map $f_s$. For generic $s$,
 we  obtain a polynomial automorphism (Lemma \ref{lem:regular S}), which satisfies the assumptions of 
   Theorem \ref{thm:mainArch} or \ref{thm:mainNA}
(Lemmas \ref{lem:transversality S} and \ref{lem:ducros}). 
Thus for every $s$, $\mathrm{Jac}(f_s)$ is of modulus 1 and we can conclude that the same holds for $\mathrm{Jac}(f)$. 
The details however require a little bit of algebro-geometric technology. 

\smallskip

Let $\pi : \A^2_S \to S$ be the natural projection map. 
Observe that for every $s\in  S$,  the fiber $\pi^{-1}(s)$ is canonically isomorphic to $\A^2_{\kappa(s)}$, 
where $\kappa(s)$ is the residue field of $s$.  
For such a $s$, we   let $C_s := \pi^{-1}(s) \cap C$ be the specialization of $C$, and likewise we
denote by $f_s$ and $f^{-1}_s$  the maps respectively induced by $f$ and $f^{-1}$  on $\A^2_{\kappa(s)}$.

\begin{lem}\label{lem:regular S}
There exists a non-empty  open subset $S' \subset S$ such that for any $s \in S'$ the map $f_s$ is a regular polynomial automorphism of degree $d$.
\end{lem}

\begin{proof}
Observe first that $f \circ f^{-1} = \id$ hence  $f_s \circ f^{-1}_s = \id$ on $\A^2_{\kappa(s)}$ so $f$ is an automorphism for all $s\in S$.

The condition for being regular of degree $d$ can be stated as follows  :
 expand $f$ as $f = f^{(0)} + f^{(1)} + \ldots + f^{(d)}$, where $f^{(i)}: \A^2 \to \A^2$ contains only homogeneous terms of degree $i$.
Then $f$ is regular of degree $d$ if and only if the composition $f^{(d)} \circ f^{(d)}$ is not identically $0$.
Now the set of $s\in S$ where $f^{(d)}_s \circ f^{(d)}_s \not\equiv 0$ is  open,  and the result follows.
\end{proof}

\begin{lem}\label{lem:transversality S}
Suppose that $p\in C(R)$ is a (closed) periodic point for $f$ lying in $C$ and defined over $R$ such that  the transversality condition (T) holds for $f$. 

Then there exists an open subset $S' \subset S$ such that for any $s \in S'$ the condition (T) is also satisfied for $f_s$.
\end{lem}

\begin{proof}
The condition that $p_s$ belongs to the  regular locus of $C_s$ is open since it is given 
by an equation of the form 
$d\phi_s (p_s) \neq 0$ where $\phi \in R[x,y]$ is an equation of $C$. 

If $p$ has exact period $k$, the condition that the period of $p_s$ is also $k$ is given by the open condition: $f (p), \ldots, f^{k-1}(p) \neq p$. 

The condition that $T_pC$ is not invariant by an iterate of $f$ is equivalent to say that 
the vector $(\frac{\partial \phi}{\partial x}, \frac{\partial \phi}{\partial y})$ is not an eigenvector for
$df^{2k}$  which is open.
\end{proof}

The key point is the following lemma. 

\begin{lem}\label{lem:ducros}
Suppose that $C$ contains infinitely many periodic points of $f$.
Then for any $s\in S$ the curve $C_s$ contains infinitely many periodic points of $f_s$.
\end{lem}

Before proving this lemma, let us show how to conclude the proof of Theorems \ref{thm:mainArch} and \ref{thm:mainNA}. 
Since $\jac(f)$ lies in $R$, it may be viewed as a regular function from $S$ to $\A^1_{\Q^{\rm alg}}$. 

Fix any embedding $\Q^{\rm alg} \subset \C$.
By Theorems \ref{thm:mainArch} (resp. \ref{thm:mainNA})
over $\Q^{\rm alg}$, we know that for any closed point $s\in S$, the algebraic number $\jac(f)(s)$ 
has all its conjugates of modulus $1$ (resp. it is a root of unity).

If $\jac(f)$ were not a constant function then its 
image would contain an open affine subset of $\A^1_{\Q^{\rm alg}}$, 
and in particular at least one algebraic number  of modulus different from $1$. 
Therefore  $\jac(f)$ is a constant lying in  $\Q^{\rm alg}$ and we are done.\qed

\begin{proof}[Proof of Lemma \ref{lem:ducros}]
For any $n\ge1$, write $f^n = (f^n_1, f^n_2)$ in affine coordinates $(x,y)$, and pick a
defining equation  $C= \{\phi =0\}$ with $\phi \in R[x,y]$.

Given an integer $l\geq 1$, denote by $\mathcal{I}_{n,l}$ the coherent ideal sheaf generated by the polynomials
$f^n_{1}-x$, $f^n_{2}-y$ and $\phi^l$. Let $\mathcal{F}_{n,l}$ be the quotient sheaf $\mathcal{O}_{\A^2} /\mathcal{I}_{n,l}$, and denote by $X_{n,l}$ the $S$-subscheme of $\A^2_S$ defined by $\mathcal{F}_{n,l}$.

We claim that the  map of schemes $\pi : X_{n,l} \to S$ is proper (hence finite since $X_{n,l}$ is a sub-scheme of $\A^2_S$). Taking  this fact for granted we proceed with the proof of the lemma.

Since $\pi: X_{n,l} \to S$ the sheaf $\mathcal{G}:= \pi_* \mathcal{F}_{n,l}$ is coherent on $S$.
It follows from Nakayama's lemma \cite[Exercice II.5.8]{hartshorne} that the function $s \mapsto \dim_{\kappa(s)} \mathcal{G}_s/\mathfrak{m}_s \mathcal{G}_s$ is upper semi-continuous. 
Now observe that $$\mathcal{G}_s/\mathfrak{m}_s \mathcal{G}_s = \oplus_{p\in \A^2_{\kappa(s)} } \kappa(s)[x,y] / ((f^n_1)_s -x, (f^n_2)_s -y, \phi^l_s)~.$$

Pick any closed point $s\in S$. By the Nullstellensatz, for $l$ large enough  and for any given point $p \in \pi^{-1}(s)$ the stalk of the coherent sheaf $\mathcal{F}_{n,l}/\mathfrak{m}_s \mathcal{F}_{n,l}$  at $p$ coincides with the stalk of $\kappa(s)[x,y] / ((f^n_1)_s -x, (f^n_2)_s -y)$. To simplify notation let us denote by 
$\mu(p,f^n_s)$ the multiplicity of $p$ as a fixed point for $f^n_s$ that is 
the dimension of $\mathcal{O}_{\A^2_\kappa(s),p} / ((f^n_1)_s -x, (f^n_2)_s-y)$.
Then we have
\begin{multline*}
\sum_{p\in C\cap \fix(f^n_s)} \mu(p,f^n_s)
\\=
\sum_{p\in C\cap \pi^{-1}(s)} \dim_{\kappa(s)} \kappa(s)[x,y] / ((f^n_1)_s -x, (f^n_2)_s -y)
\\
=
\sum_{p\in \A^2\cap \pi^{-1}(s)} \dim_{\kappa(s)} \kappa(s)[x,y] / ((f^n_1)_s -x, (f^n_2)_s -y, \phi^l_s)
\\
=
\dim_{\kappa(s)} \mathcal{G}_s/\mathfrak{m}_s \mathcal{G}_s
\ge 
\dim_{K} \mathcal{G}_\eta
=
\sum_{p\in C_K} \mu(p,f^n_K)
~,
\end{multline*}
where $\eta$ denotes the generic point of $S$, and $f_s$ (resp. $f_K$) is the map induced by $f$ on $\A^2_{\kappa(s)}$ (resp. on $\A^2_K$).

\smallskip

By assumption we know that the quantity $$\sum_{p\in C_K} \mu(p,f^n_K)\ge \card (C_K \cap \per(f^n_K))$$ tends to infinity as $n\to \infty$. It follows that $\sum_{p\in C_{\kappa(s)} \cap \fix(f^n_s)} \mu(p,f^n_s) \to \infty$. By \cite{shub-sullivan} for any $p\in \fix(f^n_s)$ the sequence of multiplicities $\{\mu(p,f^n_s)\}_{n\in \N}$ is bounded\footnote{The paper of Shub and Sullivan was written in the setting of
 differentiable mappings. 
 However the arguments work in the formal category over any field of characteristic zero. 
 The latter assumption is crucial since their result is not valid in positive characteristic.}.  
Hence $\card (C_{\kappa(s)} \cap \fix(f^n_s))$ tends to infinity, as was to be shown.

\medskip

It remains to prove that the projection map $\pi : X_{n,l} \to S$ is proper.
Let $X_n$ be the $S$-scheme defined by the equations $f^n_{1}-x$, $f^n_{2}-y$.
Since $X_{n,l}$ is a subscheme of $X_n$ it is sufficient to show that $\pi: X_n \to S$ is proper. 

To simplify notation we shall only treat the case $n=1$. Consider the intersection
of the diagonal $\Delta$ and the graph $\Gamma$ of $f$  in $\P^2_S \times \P^2_S$, and denote by $Y$ its projection into the first factor.  The projection map $Y \to S$ is projective hence proper. Observe that for each $s\in S$, $Y_s$ is the union of the support of $(X_1)_s$ which is finite by Proposition~\ref{prop:no fixed curves} together with two points at infinity $p_-(s)$ and $p_+(s)$ corresponding to the (unique) point of indeterminacy of $f_s$ and its super-attracting fixed point. 

Let $Y_+$ and $Y_-$ be the irreducible components of $Y$ such that 
$(Y_\pm)_s = p_\pm(s)$ for all $s$. Since $p_+(s)$ is super-attracting, the differential of $f_s$ at $p_+(s)$ has no eigenvalue equal to $1$, and the intersection of $\Delta$ and $\Gamma$ is transversal at $p_+(s)$. 
It follows from  the next lemma (which was 
indicated to us by A. Ducros) applied to the section $s \mapsto p_+(s)$ that  
$Y_+$ is a connected component of $Y$. Replacing $f$ by $f^{-1}$ we get that $Y_-$ is also a connected component of $Y$. 

We conclude that $\pi$ is a projective map from $X_n= Y\setminus (Y_+ \cup Y_-)$ to $S$ hence it is proper. 
\end{proof}
\begin{lem}
Let $f : Y \to S$ be a finite morphism of finite presentation and let $\sigma : S \to Y$ be a section of $f$. 
If $\mathcal{O}_{f^{-1}(s), \sigma(s)} = \kappa(s)$ for all $s\in S$ then $\sigma(S)$ is open in $Y$.
\end{lem}
\begin{proof}
Pick any $s\in S$ and let $T$ be the spectrum of the henselization of $\mathcal{O}_{S,s}$. Denote by $t$ the closed point of $T$. 
Since $T$ is henselian, the finite $T$-scheme $X\times_S T$ is a disjoint union $\coprod T_i$ of spectra of local rings. 
Pick $i_0$ such that $\sigma_T(s) \in T_{i_0}$ (here $\sigma_T$ is the section obtained from $\sigma$ by base change).

The ring $\mathcal{O}(T_{i_0})$ is a module of finite type over $\mathcal{O}(T)$. Its rank over a closed point is equal to $1$ by assumption, hence is at most $1$ at any point of $T$. Since there is a section $\sigma_T$, this rank  is actually equal to $1$  everywhere, whence 
$\sigma_T(T) = T_{i_0}$.

It follows that there exists an \'etale morphism $U\to S$ whose image contains $s$ and a decomposition $ Y\times_U S =\coprod U_i$ where each $U_i$ is finite over $U$ where $\sigma_U(U) = U_{i_0}$. The image of $U_{i_0}$ in $Y$ is an open subset $Y'$ of $Y$. By construction $Y' \subset \sigma(S)$ and $Y'$ contains $\sigma(s)$. We conclude that  $\sigma(S)$ is open. 
\end{proof}

\section{Automorphisms sharing periodic points}\label{sec:infinite}

The main purpose of this  section is to prove Theorems \ref{thm:unlikely} and \ref{thm:unlikelycomp}.

\subsection{The Bass-Serre tree of $\aut[ \A^2]$}
Let us recall briefly how the group of polynomial automorphisms of the affine plane naturally acts on a  
 tree. We refer to~\cite{lamy} for details and to~\cite{serre-amalgame} for basics on (groups acting on) trees.

Denote by $A$ (resp. $E$) the subgroup of affine (resp. elementary) automorphisms. 
The intersection $A \cap E$ consist of those automorphisms of the form
$(x,y) \mapsto ( ax + b, cy + dx +e)$ with $ac \neq 0$.

Jung's theorem states that $\aut[\A^2]$ is the free amalgamated product of $A$ and $E$ over their intersection. This means that any automorphism $f \in \aut[\A^2]$ can be written as a product
$$
f = e_1 \circ a_1 \ldots \circ a_s \circ e_s
$$
with $e_i \in E$ and $a_i \in A$, and  such a decomposition is unique 
up to replacing a product $ e \circ a$ by  $ (e \circ h^{-1}) \circ (h \circ a)$ with $h \in A \cap E$. 

\smallskip

The Bass-Serre tree  $\mathcal{T}$ of $\aut[\A^2]$ is the simplicial tree whose vertices are 
left cosets modulo $A$ or $E$. In other words we choose a set of representative $S_A$ (resp. $S_E$) of the quotient of $\aut[\A^2]$ 
under the right action of $A$ (resp. of $E$). 
Then vertices of $\mathcal{T}$ are in bijection with $\{h A\}_{h \in S_A} \cup \{h E\}_{h \in S_E}$.
An edge is joining $hA$ to $h'E$  if $h' = h \circ a$ for some $a \in A$ or 
 $h = h' \circ e$ for some $e \in E$.

We endow $\mathcal{T}$ with the unique tree metric putting length $1$ to all edges. 
The left action of an automorphism $h$ on cosets induces an action on 
 $\mathcal{T}$ by isometries. It sends any vertex of the form $hA$ (resp. $hE$) to $fh A$ (resp. to $fhE$).

An automorphism is conjugate to an affine map (resp. to an elementary map) if and only if
 it fixes a vertex of the form $hA$ (resp. $hE$) in $\mathcal{T}$.
On the other hand,   a polynomial automorphism $f$ of H\'enon type acts as a hyperbolic element of $\mathrm{Isom}(\mathcal{T})$,
 that is,  there exists
a unique geodesic\footnote{that is,  a bi-infinite path in $\mathcal{T}$.} $\geo(f)$   invariant under the action of $f$. Furthermore, 
both ends of the geodesic are fixed, and $f$ acts as a non-trivial translation on $\geo(f)$.

\subsection{Proof of Theorem~\ref{thm:unlikely}}\label{subs:unlikely algebraic}

In this section we assume that both automorphisms $f$ and $g$ are 
of Hénon type and  defined over a number field $\mathbb L$. We suppose that they share a Zariski dense subset of periodic points, and
 wish to prove that they admit a common iterate.
The proof is divided into    three steps. 
\medskip

{\bf Step 1}:  $f$ and $g$ have the same equilibrium measure at any place.

\medskip

Let us  assume that $f$ and $g$ share a  set $\set{p_m}$ of   periodic points which is Zariski dense in $\mathbb{A}^2$. 
We use a  diagonal argument to extract a subset $\set{p'_k}$ satisfying the  requirements of Theorem~\ref{thm:yuan-henon}.
For this, enumerate all irreducible curves  $(C_q)_{q\in \N}$ in $\A^2_\LL$. 
We construct an auxiliary subsequence of $(p_m)$ as follows.  Let $m_1$ be the minimal integer such that $p_{m_1} \notin C_1$, and set $p'_1 = p_{m_1}$.
Then define $m_2> m_1$ to  be the minimal integer such that $p_{m_2} \notin C_1 \cup C_2$,
and set $p'_2 = p_{m_2}$. These integers exist since the set $\{p_m\}$ is Zariski dense. 
Continuing in this way one defines recursively a sequence of periodic points 
$(p'_k)$
with the desired properties.
In particular we conclude from Theorem~\ref{thm:yuan-henon} that $\mu_{f,v} = \mu_{g,v}$ for every place $v$.

\medskip

{\bf Step 2}: $f$ and $g$ have the same set of periodic points. 

\medskip

The difficulty is that we do not assume $f$ and $g$ to be conjugate by the {\em same} automorphism to a regular map. If this happens, the conclusion follows rather directly from the work of Lamy \cite{lamy}, as we will see in Step 3. 

To overcome this problem, we   proceed  as follows. 
Fix a place $v$, and define  $K_v(f) =\{ p \in \A^{2,\an}_{\LL_v},\,  \sup_{n \in \Z} | H^n (p)|  < + \infty\}$. 

\begin{lem}\label{lem:Shilov}
For any place $v$, the set $K_v(f)$ is the largest compact set in $\A^{2,\an}_{\LL_v}$
such that 
$$
\sup_{K_v(f)} |P| = \sup_{\supp (\mu_{f,v})} |P| 
$$
for all $P\in \LL_v[\A^2]$. 
\end{lem}
\begin{proof}
Since $\supp (\mu_{f,v})\subset K_v(f)$, it is sufficient to prove 
that the supremum of $|P|$ over $K_v(f)$ is attained at a point lying in $\supp (\mu_{f,v})$.

Suppose that $P$ is a polynomial function, and pick any constant $C_0>0$ such that 
$\log (|P|/C_0) \le 0$ on  $\supp (\mu_{f,v})$. Then the function
$$\tilde{G} := \max \{ G, \log (|P|/(C_0+\varepsilon))\}$$ is a continuous non-negative function on $\A^2_{\LL_v}$ that induces a continuous semi-positive metric on $\O(1)$.
Since $\tilde{G} = G$ near $\supp (\mu_{f,v})$, from Corollary~\ref{cor:locality} we deduce
 the equality of measures $\MA( \tilde{G}) = \MA(G)$,
and it follows from Yuan-Zhang's theorem, see~\cite{YZ13a} that 
$\tilde{G} - G$ is a constant, hence $\tilde{G} = G$. It follows that 
$\log (|P|/(C_0+\varepsilon)) \le 0$ on $K_v(f)$. By letting $\varepsilon\to 0$ we conclude that 
$\log (|P|/C_0) \le 0$ on  $K_v(f)$.
\end{proof}

In plain words, the polynomially convex envelope of $\supp (\mu_{f,v})$ is the set $K_{f,v}$.
Since $\mu_{f,v} = \mu_{g,v}$ for all $v$ we conclude that  $K_{f,v} = K_{g,v}$. 

Now pick any periodic point $p$ of $f$. At the place 
 $v$, it belongs to $K_{f,v}$ hence to $K_{g,v}$. Since Lee's height can be computed by summing the local quantities
$G_{g,v} := \max \{ G^+_{g,v}, G^-_{g,v}\}$ and since  $\{G_{g,v} = 0\} = K_{g,v}$
we get that the canonical $g$-height of $p$ is zero, hence $p$ is $g$-periodic.

\medskip

In the sequel we   actually need a stronger information. 

\begin{lem}\label{lem:interm1}
Suppose $f$ and $g$ are two hyperbolic polynomial automorphisms of the affine plane
defined over a number field $\LL$, satisfying the assumptions of Theorem~\ref{thm:mainNA}.

Then for all places $v$ over $\LL$, and for any H\'enon-type
 automorphism $h$ belonging to the subgroup generated by $f$ and $g$, one has  $K_{h,v} = K_{g,v} = K_{f,v}$. 
\end{lem}

\begin{proof}
We already know that  $K_v := K_{g,v} = K_{f,v}$. 
Since this compact set is invariant by both $f$ and $g$, it follows that $h$ also preserves $K_v$,
and this implies $K_v \subset K_{h,v}$ for all $v$.
Now let $F_n$ denote the set of points of period $n$ for $f$. 
For all $v$, we have $F_n \subset K_{h,v}$ hence the canonical $h$-height of $F_n$ is equal to $0$. Extracting a subsequence if necessary, we may always assume that $F_n$ is generic since the set of periodic points of a hyperbolic automorphism is Zariski-dense. 
By Yuan's theorem  $F_n$ is equidistributed with respect to the equilibrium measure of both $K_v$ and $K_{h,v}$ and we conclude that $K_{h,v} = K_v$.
\end{proof}

\medskip

{\bf Step 3}: $f$ and $g$ admit a common iterate.

\medskip

We use Lamy's structure theory of subgroups of the group of polynomial automorphisms of the plane~\cite{lamy}. 
Let us assume by contradiction that $f$ and $g$ admit no common iterate. 

\begin{lem}\label{lem:interm2}
Under the above hypotheses, 
there exists two H\'enon-type  elements $h_1, h_2$ in the subgroup generated by $f$ and $g$ 
and a polynomial automorphism $\varphi$ such that 
\begin{itemize}
\item
the subgroup $H$ generated by $h_1$ and $h_2$ is a free non-abelian group;
\item
any element in $H$ that is not the identity is of Hénon-type;
\item 
for any element $h\in H$,  the automorphism $\varphi^{-1} \circ h \circ \varphi$ is a regular automorphisms of $\A^2_L$.
\end{itemize}
\end{lem}
\label{step3}
The rest of the argument is now contained in~\cite[Th\'eor\`eme~5.4]{lamy}. We
include it for the convenience of the reader.
We may assume that $h_1, h_2$ are two regular polynomial automorphisms of $\A^2_L$.
By Lemma~\ref{lem:interm1} we have $K_v:= K_{h_1,v} = K_{h_2,v}$ at all places. Pick any 
archimedean place $v$. Then  $\mu:= \mu_{h_1,v} = \mu_{h_2,v}$. 
Now since both $h_1$ and $h_2$ are regular, it follows that $G_1:= \max \{ G^+_{h_1},  G^-_{h_1}\}$ and $G_2$ are both equal to the Siciak-Green function of $K_v$ by~\cite[Proposition~3.9]{BS1}, and are hence equal.

Replacing  $h_1$ by its inverse if necessary, it follows that  $G^+_{h_1} = G^+_{h_2}$ on a 
non-empty open set where the two  functions are positive. 
Since   these functions are pluriharmonic where there are non zero, and since for a H\'enon-type automorphism $h$, 
 $\cd\setminus K^+$ is connected (indeed $\cd \setminus K^+ = \bigcup_{n\geq 0} h^{-n} (V^+)$)\label{connected}
 we deduce that they coincide everywhere. 
We conclude that  the positive closed $(1,1)$-currents $T:=  dd^c  G^+_{h_1}=  dd^c  G^+_{h_2}$ are equal. 
Now consider    the commutator $ h_3 = h_1 h_2 h_1^{-1} h_2^{-1}$.
Observe that  $h_3^* T = T$, and $h_3$ is regular by the previous lemma. 
Since the support of $T$ has a unique point on the line at infinity, replacing $h_3$ by $h^{-1}_3$ if needed 
we may suppose 
that this point is not an indeterminacy point of $h_3$. It then    follows that the mass of $h^*_3T$ equals
the degree of $h_3$ times the mass of $T$ which is contradictory. This completes the proof of Theorem \ref{thm:unlikely}.

\begin{proof}[Proof of Lemma~\ref{lem:interm2}]
We   need the following two facts.
Let $f$ and $g$ be  two polynomial automorphisms of H\'enon type. 
\begin{itemize}
\item[(F1)]
given a segment $\gamma$ in the Bass-Serre tree $\mathcal{T}$, there exists a polynomial automorphism $\varphi$ such that any polynomial automorphism $f$ such that if $\geo(f)$ contains $\gamma$
then $\varphi^{-1} \circ f \circ \varphi$ is regular\footnote{This follows from the fact that an automorphism of H\'enon type whose associated geodesic contains the edge $\id S$ in the terminology of Lamy is regular. Indeed such a map is cyclically reduced by~\cite[Remarque~2.3]{lamy}, hence a composition of generalized H\'enon maps by~\cite[Theorem~2.6]{friedland-milnor}, hence regular.}.
\item[(F2)]
If $f$ and $g$ do not share a non-trivial iterate, then 
  the two geodesics $\geo(f)$ and $\geo(g)$
have a compact intersection (possibly empty). See~\cite[Proposition~4.10]{lamy}.
\end{itemize}

Let us first treat the case when $\geo(f) \cap\geo(g)$ contains a segment. 
Pick an edge $\gamma$
in this intersection. By (F1) we may assume $f$ and $g$ are regular automorphisms.

Set then $h_1 = f^N$ and $h_2 = g^N$ where $N$ is greater than the diameter of   $\geo(f) \cap \geo(g)$.
The invariant geodesics of these two automorphisms are equal to 
$\geo(f)$ and $\geo(g)$ respectively. Now pick 
any non-trivial  word  $$h = h_1^{n_p} \circ h_2^{m_p} \circ \ldots \circ h_1^{n_1} \circ h_2^{m_1}$$
with $p \ge 1$ and all $n_i, m_i \in \Z\setminus\set{0}$.
Then a   ping-pong argument (see~\cite[Proposition~4.3]{lamy}) shows that 
$\gamma$ lies in the interior of the segment $[h(\gamma), h^{-1}(\gamma)]$
which implies that $h$ is  of H\'enon type and $\geo(h) \supset\gamma$. By (F1) we conclude that
$h$ is also regular. 
This concludes the proof in this case.

\medskip

Assume now that $\geo(f)$ and $\geo(g)$ are either disjoint, or is reduced to a singleton.
Then there exists a unique segment $I= [\gamma_1, \gamma_2]$ in the tree
with $I\cap \geo(f) = \{ \gamma_1\}$ and $I\cap \geo(g) = \{ \gamma_2\}$.
Pick any element $\gamma \in I$, and $N$ large enough
such that the translation length of both $f^N$ and $g^N$ are larger than twice the diameter of $I$. Then both automorphisms $h_1
= f^N g^N$ and $h_2 = f^{2N} g^{2N}$
satisfy $\gamma \in [h_i(\gamma), h_i^{-1}(\gamma)]$ so that we can apply the same argument as in the previous case. 
The proof is complete.
\end{proof}

\subsection{Proof of Theorem~\ref{thm:unlikelycomp}}

 Here we assume that  $f$ and $g$ are automorphisms of Hénon type with \emph{complex} coefficients 
such that  $|\jac(f)|\neq 1$, that  share an infinite set $\mathcal{P}$ of   periodic points.

\smallskip

As in Section \ref{sec:general fields}, we pick a  
finitely generated $\Q^{\rm alg}$-algebra $R$, such that $f$ and $g$ are defined over $K:={\rm Frac}(R)$, and we set $S = \spec (R)$.  As before we use the notation $f_s$ for the specialization of $f$ at $s\in S$.

We proceed by contradiction, so let us assume that $f$ and $g$ do not have any common iterate. 
By Lemma~\ref{lem:interm2} we can fix two elements $h_1, h_2$ in the subgroup generated by $f$ and $g$ such that any non trivial element in $H := \langle h_1,h_2\rangle$ is  of Hénon-type and regular.

\begin{lem}
There exists an open subset $S' \subset S$ such that for any $s \in S'$ the maps  
$h_{1,s}$, and $h_{2,s}$  as well as 
 their commutators $(h_1h_2 h_1^{-1}h_2^{-1})_s$, $(h_1h^{-1}_2 h_1^{-1}h_2)_s$ are regular polynomial automorphisms of  Hénon type.
\end{lem}

\begin{proof}
The condition for an automorphism $h$ to be regular is    open   since it amounts to 
saying that the indeterminacy loci of $h$ and $h^{-1}$ are disjoint. 
A theorem of J.-P. Furter~\cite{furter} asserts 
that an automorphism $h$ is of Hénon type if and only if $\deg(h^2) > \deg(h)$. Since there exists an open set where $\deg(h_s) = \deg(h)$ and  $\deg(h_s^2) = \deg(h^2)$, the result follows.
\end{proof}

Thus, replacing $S$ by a Zariski dense open subset, we may assume that 
for all $s\in S$, the maps $h_{1,s}$, $h_{2,s}$, $(h_1h_2 h_1^{-1}h_2^{-1})_s$ or $(h_1h^{-1}_2 h_1^{-1}h_2)_s$ are regular polynomial automorphisms of  Hénon type. 

\smallskip

Let us now pick    $s\in S(\Q^{\rm alg})$ such that $\jac(f_s)$ possesses at least one complex conjugate of norm $\neq 1$.
It is always possible to find such a parameter since otherwise $\jac(f)$ would be an algebraic number
whose complex conjugates would all lie on the unit circle, contradicting our assumption. 

\smallskip

Our next claim  is that  $\mathcal{P}_s$ is infinite. Indeed assume by contradiction that $\mathcal{P}_s$ is finite
so that $\mathcal{P}_s$ is included in the set of fixed points of $f^{n_0}$. 
For each $n\ge n_0$ denote by $X_n$ the subscheme of $\A^2_S$ whose underlying space is 
$$\mathcal{P} \cap \set{(x,y), \ f^n(x,y) = (x,y)},$$ endowed with the scheme structure 
induced by the quotient sheaf $$\mathcal{O}_{\A_S^2}/ (f^n_1-x, f^n_2-y),$$ where $f^n=(f_1^n,f_2^n)$.

For any $p\in \mathcal{P}_s$, the ordinary multiplicity $e(p,X_{n,s})$ of $p$ as a point in $X_{n,s}$
 is equal to the multiplicity $\mu(p,f^n_s)$ as a fixed point for $f^n_s$. By the Shub-Sullivan 
 Theorem~\cite{shub-sullivan}, the sequence $e(p,X_{n,s})$ is bounded, hence
$\sum_{p\in \mathcal{P}_s}e(p,X_{n,s})$ is bounded.
 
Arguing as in the proof of Lemma~\ref{lem:ducros}, we see that the map $X_n \to S$ is proper and finite, and
by Nakayama's lemma we get that 
$$
\# \left[\mathcal{P} \cap \{f^n = \id\}\right]
\le 
\sum_{p\in \mathcal{P}}e(p,X_n)
\le 
\sum_{p\in \mathcal{P}_s}e(p,X_{n,s}) <+\infty~.
$$
Now observe that $\mathcal{P} \cap \{f^{n!} = \id\}$ contains all the periodic points in $\mathcal{P}$ of period $\le n$ 
so it follows that  the cardinality of this set tends to infinity as $n\to\infty$. This  is  contradictory, thereby
showing  that $\mathcal{P}_s$ is infinite.

\smallskip

To conclude the proof of the theorem, fix an Archimedean place $v$ at which $|\jac(f_s)|_v \neq 1$. 
Since $\mathcal{P}_s$ is infinite, Theorem~\ref{thm:mainArch} implies $\mathcal{P}_s$ to be Zariski dense.
By Lemma~\ref{lem:interm1} we get that $h_1$ and $h_2$ have the same $K_v$, and by arguing as in Step~3 on p.\pageref{step3},
we conclude that the two pairs of functions $\{G^+_{h_{1,s}}, G^-_{h_{1,s}}\}$ and $\{G^+_{h_{2,s}}, G^-_{h_{2,s}}\}$ are identical. 
It follows that one among the two commutators $(h_1h_2 h_1^{-1}h_2^{-1})_s$ or $(h_1h^{-1}_2 h_1^{-1}h_2)_s$, denoted by $h_3$, 
 leaves invariant $G^+_s := G^+_{h_{1,s}}$, that is $G_s\circ h_3 = h_3$. As we saw in Theorem~\ref{thm:unlikely}, 
  since  $h_3$ is a regular automorphism, this property cannot be true. This contradiction finishes the proof. \qed

\begin{rem}
The argument uses in an essential way the fact that the place $v$ is  Archimedean. 
Indeed we ultimately rely on the fact that 
for any two regular maps of H\'enon type  the equality $G_{h_1} = G_{h_2}$ forces that of   
$\{G^+_{h_{1}}, G^-_{h_{1}}\}$ and $\{G^+_{h_{2}}, G^-_{h_{2}}\}$. 
The  corresponding statement over a non-Archimedean field $k$ is not true. 

Indeed as above  it can be shown that  
$\{G^+_{h_{1}}, G^-_{h_{1}}\}=\{G^+_{h_{2}}, G^-_{h_{2}}\}$
if and only if $h_1$ and $h_2$ admit a common 
iterate. On the other hand if $h$ is any regular automorphism such that 
 $h$ and $h^{-1}$ have   their coefficients in the ring of integers of $k$, 
then $G_h (x,y) = \log \max \{1, |x|, |y|\}$.
\end{rem}

\subsection{Sharing cycles}\label{sec:cycles}
Let us start with  the following observation.

\begin{prop}\label{prop:cycles}
Let $f$ be an  automorphism of  H\'enon type defined over a number field $\LL$. Let $(F_m)$ 
be any sequence of disjoint periodic {\em cycles}.
Then the sequence of probability measures $(\mu_m)$   equidistributed over the Galois conjugates of $F_m$ converges weakly 
to $\mu_{f,v}$ for all places $v$.
\end{prop}

As a consequence we have the following variation on Theorem \ref{thm:unlikely}.

\begin{cor}\label{cor:unlikely cycles}
 Let $f$ and $g$ be two      polynomial automorphisms of H\'enon type 
of the affine plane, defined over a number field. 

Then if  $f$ and $g$ share an infinite set of  periodic cycles, there exist two non-zero integers $n,m\in \Z$ such that $f^n = g^m$.
\end{cor}

Indeed to prove the corollary it suffices to repeat the proof of Theorem \ref{thm:unlikely} starting from Step 2. 

\begin{proof}[Proof of Proposition \ref{prop:cycles}]
We may assume that all $F_m$ are Galois invariant.
The result does not quite follow from Lee's argument of~\cite[Theorem~A]{lee} 
since we do not assume the set $\bigcup_m (F_m \cap C)$ to be finite for every curve $C$.
We claim however that any algebraic curve $C$ it holds that 
\begin{equation}\label{eq:almost}
\# (C \cap F_m) = o(\# F_m) \text{ as } n\to \infty.
\end{equation}
One then argues exactly as in  \cite[Proof of Theorem~1]{favre-gauthier} 
to conclude that $\mu_m$ converges to $\mu_f$.

\smallskip

Let us justify~\eqref{eq:almost}. Suppose by contradiction that there exists $\varepsilon >0$ such that 
$$\# (C \cap F_m)  \ge \varepsilon\, \# F_m~.$$
First observe that the minimal period of all points in $F_m$ tends to infinity since 
$f^n$ admits only finitely many fixed points for any $n>0$.
Pick any integer $N> 1/\varepsilon$ and $n$ large enough such that the periods of all points in $F_m$ are larger than $N$. We claim that
$$
\# \{ p \in F_m \cap C, \, \text{ such that } f^k(p) \in C \text{ for some } 0 <k \le N\}
$$
tends to infinity. But this implies that  $C \cap f^{-k}(C)$ is infinite for some $0 < k \le N$, whence $ f^k(C) = C$, a contradiction. 

To prove the claim, let $B$ denote the set of points in $F_m\cap C$ such that 
$f^k(p) \notin C$ for all $1\le k \le N$, and let $G$ be its complement in $F_m\cap C$.
We want to estimate $\# G$.
For this, we see that  $\# B \times N \le \# F_m$, hence
$$\# G \ge \# (F_m \cap C)  - \# B \ge \varepsilon \# F_m - \frac1{N} \# F_m \to \infty$$
as required.
\end{proof}


\section{Reversible polynomial automorphisms}\label{sec:reversible}

A polynomial automorphism of $\mathbb A^2$ is said to be {\em reversible} if there exists a polynomial automorphism $\sigma$, which may or may not be an involution,  such that $\sigma^{-1} f \sigma  = f^{-1}$. Any such $\sigma$ is then called a {\em reversor}.

Since invariance under time-reversal appears frequently
in physical models, such mappings have attracted a lot of  attention in the mathematical physics literature.
In the context of  plane polynomial automorphisms, reversible mappings were classified by Gomez and Meiss in \cite{gm1, gm2}. 
In particular they prove that the reversor $\sigma$ is either affine or elementary and of finite (even) order. Moreover they show
 that when $\sigma$ admits a curve of fixed points, then $\sigma$ must be an  involution conjugate to the affine involution $t:(x,y)\mapsto (y,x)$.

Our aim is to prove the following:
\begin{prop}\label{prop:reversible}
Suppose that $f$ is a reversible polynomial automorphism of H\'enon type and that $\sigma$ is an involution conjugating $f$ to $f^{-1}$. 

Then any curve of fixed points of $\sigma$ contains infinitely many periodic points of $f$.
\end{prop}

Specific examples include all H\'enon transformations of Jacobian 1, that are of the form $(x,y)\mapsto ( p(x) - y, x)$, for which the reversor is the affine involution  $t$. So is  the H\'enon mapping $(x,y)\mapsto (-y, p(y^2)-x)$,  of Jacobian $-1$. More generally, 
 a mapping of the form $t  H^{-1} t H$ is reversible with reversor $t$, where $H$ denotes any polynomial automorphism.

Let us also observe that taking iterates is really necessary in Conjecture \ref{conj:MM2}. Indeed 
pick for some $n\geq 2$ a  primitive $n$-th root of unity $\zeta$,  and let $p$ be any polynomial such that
$p(\zeta x) = \zeta p(x)$. Then the   automorphisms $ ( p(x) - y, x)$, $(\zeta x, \zeta y)$  commute and the automorphism 
 defined by $H := ( p(x) - y, x) \circ (\zeta x, \zeta y)$ is not reversible but its $n$-th iterate is. 
Observe also that the jacobian of $H$ equals $\zeta^2$.

\begin{proof}[Algebraic proof of Proposition~\ref{prop:reversible}]
As observed above, we may assume that $\sigma = t$ so that the curve of fixed points is actually the diagonal $\Delta = \{ x=y\}$.

Observe now that any point $ p\in \Delta \cap f^n (\Delta)$ satisfies
$f^{-n}(p) = \sigma f^n \sigma(p) = \sigma f^n (p) = f^n(p)$, and is thus
periodic of period $2n$.

To conclude, it remains to prove that $\# \Delta \cap f^n (\Delta) \to \infty$.
For any $p \in \Delta \cap f^n (\Delta)$, we denote by $\mu_n(p)$ the multiplicity of intersection of 
$\Delta$ and $f^n(\Delta)$ at $p$. We rely on the following result of Arnol'd~\cite{arnold93}, see also \cite{SY13}.

\begin{lem}\label{lem:bdd}
For any $p\in \Delta$, the sequence $\mu_n(p)$ is bounded.
\end{lem}

Now since $f$ is a polynomial automorphism of H\'enon type we may choose
affine coordinates such that $f$ extends to $\P^2$ as a regular map. 
Recall that $f$ admits a super-attracting point $p_+$ and a point of indeterminacy $p_-$
on the line at infinity and that $p_+ \neq p_-$.
Write $\overline{\Delta}$ for the closure of the diagonal in $\P^2$.

By~\cite{BS1}, we know that there exists an integer $n_0\ge 1$ such that $f^n (\overline{\Delta})\ni p_+$  for all  $n\ge n_0$ and
 $f^n (\overline{\Delta})\ni p_-$  for all  $n\le -n_0$. 
It follows that for  all $n\ge n_0$, the intersection 
$f^{-n_0} (\overline{\Delta}) \cap f^n (\overline{\Delta})$ is included in $\A^2$.
By B\'ezout's Theorem  we infer that 
$$
\sum_{p\in \Delta} \mu_n(p)
= 
f^{-n_0} (\overline{\Delta}) \cdot  f^{n-n_0} (\overline{\Delta})
= \deg( f^{-n_0} (\overline{\Delta})) \times \deg ( f^{n-n_0} (\overline{\Delta}) ) \to \infty~.
$$
 We conclude that there are infinitely many fixed points of $f$ on $\Delta$ for otherwise their multiplicities would have to grow to infinity, contradicting Lemma \ref{lem:bdd}. 
\end{proof}

\begin{proof}[Analytic proof of Proposition~\ref{prop:reversible}]
Let us sketch an alternate argument for the fact that $\# \Delta \cap f^n(\Delta)\cv \infty$, which is based on intersection theory of laminar 
currents. 

For notational ease, let us assume that $n = 2k$ is even. Then $\# \Delta \cap f^n(\Delta) = \# f^{-k}(\Delta) \cap f^k(\Delta)$. 
We know from \cite{BS1} that the sequence of positive closed $(1,1)$-currents $d^{-k} [f^k(\Delta)]$ converges to $T^-$, and likewise for $T^+$.  It follows from \cite{bs5} that this 
convergence holds in a geometric sense. Informally this means that 
 we can discard a part of $d^{-k} [f^k(\Delta)]$  of arbitrary small mass, uniformly in $k$, 
 such that the remaining part is made of disks of uniformly bounded geometry, which geometrically converge to the disks making up the laminar structure of $T^-$.  
 
 To state things more precisely, we follow the presentation of \cite{isect}. Fix $\e>0$. Given a generic subdivision $\mathcal Q$ of $\cd$ by affine cubes of size $r>0$, there exists uniformly laminar currents $T_{\mathcal{Q}, k} ^- \leq d^{-k} [f^k(\Delta)]$ and $T_{\mathcal{Q}, k} ^+ \leq d^{-k} [f^{-k}(\Delta)]$ made of graphs in these cubes and such that the mass of $d^{-k} [f^{\pm k}(\Delta)] - T_{\mathcal{Q}, k} ^\pm$ is bounded by $Cr^2$ for some constant $C$ \cite[Prop. 4.4]{isect}.
  Therefore, up to extracting a subsequence,  
 the currents $T_{\mathcal{Q}, k} ^\pm$ converge to  currents  $T_{\mathcal{Q}} ^\pm \leq T^\pm$ such that 
 $\mathbf{M} (T^\pm - T_{\mathcal{Q}} ^\pm)\leq Cr^2$. Then we infer from \cite[Thm. 4.2]{isect} that if $r$ is smaller than some $r(\e)$, 
 $\mathbf{M}(T^+\wedge T^-  - T^+_{\mathcal Q} \wedge T^-_{\mathcal Q}) \leq \e/2$. 
 Furthermore, only transverse intersections account for  the wedge product  
$T^+_{\mathcal Q}  \wedge T^-_{\mathcal Q}$.  If we denote by $\geom$ the geometric intersection product for curves, which consists in putting a Dirac mass at any proper intersection, without counting multiplicities, we have the weak convergence
 $$T^+_{\mathcal Q, k} \geom T^-_{\mathcal Q, k} \cv T^+_{\mathcal Q}  \geom T^-_{\mathcal Q} = T^+_{\mathcal Q} \wedge 
  T^-_{\mathcal Q},$$ where the last equality follows from \cite[Thm 3.1]{isect}. Therefore 
 the mass of $T^+_{\mathcal Q, k} \geom T^-_{\mathcal Q, k}$ is larger than $1-\e$ for large enough $k$, which was the result to be proved. 
\end{proof}

\begin{rem}
Observe that the analytic argument implies that
$\Delta$ intersects $f^{2k}(\Delta)$ transversally at $\sim d^k$ points. 
We claim that this implies
$$\# {\rm Per} (f^{2k}) \cap \Delta = d^k ( 1 + o(1)) ~,$$
that is, most of $\Delta\cap f^{2k}(\Delta)$ is made of points of exact period $2k$. Indeed, assume that this is not the case. Then there exists $\e>0$ and a sequence $k_j\cv\infty$  such that $ \e d^{k_j}$ of these points have a period which is a proper divisor $N$ of $d^{2k_j}$, in particular $N\leq d^{k_j}$. Let $F_j$ be this set of points and $\nu_j  = d^{-k_j} \sum_{p\in F_j} \delta_p$. 
By \cite{bls2}, the measure equidistributed on $\mathrm{Fix}(f^{k_j})$ (which has cardinality $d^{k_j}$) converges to $\mu_f$. 
 Thus any cluster limit $\nu$ of $(\nu_j)$ has mass at least $\e$ and satisfies $\nu\leq  \mu_f$. It follows that $\mu_f$ gives a mass at least
  $\e$ to $\Delta$, which is contradictory.  
\end{rem}


\appendix

\section{A complement on the non-Archimedean Monge-Amp\`ere operator}
Let $K$ be any complete non-Archimedean field.  We prove
\begin{thm}\label{thm:locality}
Suppose $L\to X$ is an ample line bundle over a smooth $K$-variety of dimension $d$. Pick any two semi-positive continuous metrics $|\cdot|_1, |\cdot|_2$ in the sense of Zhang.
Here we establish the following theorem.  
\begin{equation}\label{eq:compar}
  \one_{\{|\cdot|_1<|\cdot|_2\}}\,c_1( L,\min\{|\cdot|_1,|\cdot|_2\})^d = 
  \one_{\{|\cdot|_1<|\cdot|_2\}}\,c_1( L,|\cdot|_1)^d.
  \end{equation}
\end{thm}
As in \cite[\S 5]{nama} this result implies 
\begin{cor}\label{cor:locality}
Suppose $L\to X$ is an ample line bundle over a smooth $K$-variety of dimension $d$. Pick any two semi-positive continuous metrics $|\cdot|_1, |\cdot|_2$ in the sense of Zhang
and suppose that they
coincide on an open set $\Omega$ in the analytification of $X$ in the sense of Berkovich. 

Then the positive measures $c_1( L, |\cdot|_1)^d$ and $c_1( L, |\cdot|_2)^d$ coincide in $\Omega$.
\end{cor}
Recall that in the main body of the text, we deal  with metrics $|\cdot|$ on $\O(1) \to \P^d$ and the evaluation of the section $1$ on the analytification of the affine space $\A^d \subset \P^d$ defines a function $G := \log |1|$.
With this identification, one has
 $\MA (G) = c_1(\O(1), |\cdot|)^d$.

\begin{proof}
Assume first that the two metrics $|\cdot|_i$  are model metrics. This means that we 
can find a model $\mathfrak{X}$ of $X$ over $\spec \O_K$, and nef line bundles
$\mathfrak{L}_i \to \mathfrak{X}$ whose restriction to the generic fiber of $\mathfrak{X}$ is $L$.

In that case $c_1( L,\max\{|\cdot|_1,|\cdot|_2\})^d$ and $c_1( L,|\cdot|_1)^d$
 are both atomic measures, supported on 
divisorial points corresponding to irreducible
components of the special fiber.

 If $E$ is such a component for which $|\cdot|_1/| \cdot|_2 (x_E) <1$ 
then $|\cdot|_1/| \cdot|_2 (x_F) \le1$ for all irreducible components
  $F$ of the special fiber intersecting $E$.
  It follows that $\mathfrak{L}_1|_E=\mathfrak{L}_2|_E$ as numerical classes on $E$, 
  and hence $c_1( L,\min\{|\cdot|_1,|\cdot|_2\})^d\{ x_E\}$ and $c_1( L,|\cdot|_1)^d\{ x_E\}$ by the 
  definition of Monge-Amp\`ere measures of model functions. 
 
  \smallskip

In the general case, we may assume that we have sequences of model metrics $|\cdot|_{i,n}$
on $L$ such that $|\cdot|_{i,n} \to |\cdot|_{i}$ uniformly on $X^{\an}$.
Observe that $\Omega:=\{|\cdot|_1<|\cdot|_2\}$ is open since both metrics are continuous.
It suffices to prove that $$\int h\, c_1( L,\min\{|\cdot|_1,|\cdot|_2\})^d =\int h\, c_1( L,|\cdot|_1)^d$$  for all continuous functions $h$  whose support is contained in
  $\Omega$ and such that $0\le h\le1$.

  Pick $\e>0$ small and rational and write $\Omega_n:=\{|\cdot|_{1,n} e^{-\e}<|\cdot|_{2,n}|\}$.
  For $n\gg0$, we have $\Omega\subseteq\Omega_n$.
  Since $|\cdot|_{1,n} e^{-\e}$ and $|\cdot|_{2,n}|$ are both model metrics, we have 
 $c_1( L,\min\{|\cdot|_{1,n}e^{-\e},|\cdot|_{2,n}\})^d = 
 c_1( L,|\cdot|_{1,n})^d$ on $\Omega_n$ by the previous step.
Since $h$ is supported in $\Omega \subset \Omega_n$ we get 
$$
\int h\, c_1( L,\min\{|\cdot|_{1,n}e^{-\e},|\cdot|_{2,n}\})^d
=
\int h\, c_1( L,|\cdot|_{1,n}e^{-\e})^d
=
\int h\, c_1( L,|\cdot|_{1,n})^d
$$
for all $n$, and we conclude letting $n\to \infty$ and $\e \to 0$ and using \cite[Proposition 2.7]{chambert06}.
 \end{proof}



\end{document}